\documentclass[12pt,reqno]{amsart}

\numberwithin{equation}{section}

\usepackage{amssymb}
\usepackage{amsmath}
\usepackage{amsbsy}
\usepackage{amscd}
\usepackage{amsthm}
\usepackage{amsfonts}
\usepackage{enumitem}
\usepackage{ifpdf}
\ifpdf \usepackage[pdftex,pdfstartview=FitH,pdfpagemode=none,colorlinks,bookmarks,linkcolor=blue]{hyperref} \else  \usepackage[hypertex]{hyperref} \fi
\usepackage[nobysame,abbrev,alphabetic]{amsrefs}
\usepackage{color}

\newtheorem{theorem}{Theorem}[section]
\newtheorem{condition}[theorem]{Condition}

\newtheorem{lemma}[theorem]{Lemma}
\newtheorem{corollary}[theorem]{Corollary}

\newtheorem{proposition}[theorem]{Proposition}

\theoremstyle{definition}
\newtheorem{definition}[theorem]{Definition}

\newtheorem{remark}[theorem]{Remark}

\newcommand{\cL}{\mathcal{L}}

\newcommand{\cO}{\mathcal{O}}

\newcommand{\bC}{\mathbb{C}}

\newcommand{\bR}{\mathbb{R}}
\newcommand{\bZ}{\mathbb{Z}}
\newcommand{\bQ}{\mathbb{Q}}
\newcommand{\bF}{\mathbb{F}}

\newcommand{\bN}{\mathbb{N}}

\newcommand{\bT}{\mathbb{T}}

\newcommand{\bfm}{\mathbf{m}}
\newcommand{\bfn}{\mathbf{n}}

\newcommand{\bfv}{\mathbf{v}}
\newcommand{\bfw}{\mathbf{w}}
\newcommand{\bfx}{\mathbf{x}}
\newcommand{\bfy}{\mathbf{y}}
\newcommand{\bfz}{\mathbf{z}}
\newcommand{\bfzero}{\mathbf{0}}
\newcommand{\bfxi}{\boldsymbol{\xi}}
\newcommand{\bfgamma}{\boldsymbol{\gamma}}

\newcommand{\GL}{\operatorname{GL}}

\newcommand{\re}{\operatorname{Re}}

\newcommand{\rank}{\operatorname{rank}}
\newcommand{\dist}{\operatorname{dist}}
\newcommand{\Arg}{\operatorname{Arg}}

\newcommand{\Tr}{\operatorname{Tr}}
\newcommand{\Log}{\operatorname{Log}}

\newcommand{\Stab}{\operatorname{Stab}}
\newcommand{\sign}{\operatorname{sign}}

\newcommand{\imag}{\mathrm{i}}

\newcommand\diag[1]{\operatorname{diag}\left(#1\right)}

\newcommand{\onto}{\xymatrix{\ar@{>>}[r]&}}
\newcommand{\da}[4]{\xymatrix{#1 \ar@<.5ex>[r]^{#2} \ar@<-.5ex>[r]_{#3} & #4}}

\newcounter{subconst}[subsection]

\newcounter{const}

\newcounter{CONST}

\begin{document}

\title[Topological self-joinings of Cartan actions]{Topological self-joinings of Cartan actions by toral automorphisms}
\author[E. Lindenstrauss]{Elon Lindenstrauss}
\address{ Hebrew University, 91904 Jerusalem, Israel}
\author[Z. Wang]{Zhiren Wang}
\address{ Princeton University, Princeton, NJ 08544, USA}
\setcounter{page}{1}
\begin{abstract}We show that if $r\geq 3$ and $\alpha$ is a faithful $\bZ^r$-Cartan action on a torus $\bT^d$ by automorphisms, then any closed subset of $(\bT^d)^2$ which is invariant and topologically transitive under the diagonal $\bZ^r$-action by $\alpha$ is homogeneous, in the sense that it is either the full torus $(\bT^d)^2$, or a finite set of rational points, or a finite disjoint union of parallel translates of some $d$-dimensional invariant subtorus. A counterexample is constructed for the rank~$2$ case.\end{abstract}
\thanks{E.~L.~acknowledges the support of the ERC (AdG Grant 267259), the ISF (983/09) and the NSF (DMS-0800345)}
\maketitle
\small\tableofcontents

\section{Introduction}

\subsection{Background}

One of the principle results in Furstenberg's seminal paper \cite{F67} is that on the one torus $\bT = \bR / \bZ$ there are very few closed sets invariant under the natural action of higher rank multiplicative semigroups of the natural numbers,  such as the multiplicative semigroup generated by 2 and 3: indeed, any such closed and invariant set is either $\bT$ or is a finite set of rational points.
Berend extended this important result to actions of semigroups of toral endomorphisms on higher dimensional tori. We identify the automorphisms of the $d$-dimensional torus $\bT ^ d$  with the corresponding element of $\GL (d, \bZ)$. A toral automorphism $\phi \in \GL(d,\bZ)$ of $\bT^d$ is {\it irreducible} if it leaves invariant no non-trivial proper subtorus of $\bT^d$, or equivalently its characteristic polynomial is irreducible over $\bQ$. $\phi$ is {\it totally irreducible} if $\phi^n$ is irreducible for all non-zero integer $n$.
Under this identification, a {faithful $\bZ^r$-action} on $\bT^d$ by automorphisms $\alpha:\bZ^r\curvearrowright \bT^d$ can be identified with a group embedding $\alpha:\bfn \mapsto \alpha^\bfn$ of $\bZ^r$ into $\GL(d,\bZ)$.

In \cite{B83} Berend gave a necessary and sufficient condition for an abelian semigroup of automorphisms of $\bT ^ d$ to share the property of higher rank multiplicative semigroups of $\bN$ that any closed invariant set is either finite or everything. In the important special case of group actions, his results reduces to the following:

\begin{theorem}\cite{B83}\label{Berend} Let $\alpha:\bZ^r\hookrightarrow\GL(d,\bZ)$ be a faithful $\bZ^r$-action on $\bT^d$ by toral automorphisms that satisfies:

\begin{enumerate}[label=\textup{(}\roman*\textup{)},leftmargin=*]
\item $r\geq 2$;

\item $\exists\bfn\in\bZ^r$ such that $\alpha^\bfn$ is a totally irreducible toral automorphism;

\item If $v\in \bC^d$ is a non-zero common eigenvector of the commutative group of matrices $\alpha(\bZ^r)$, then $\exists\bfn\in \bZ^r$ such that $|\alpha^\bfn.v|>|v|$.
\end{enumerate}

Then the only infinite $\alpha$-invariant closed subset of $\bT^d$ is $\bT^d$ itself. In particular, $\{\alpha^\bfn.y: \bfn\in\bZ^r\}$ is dense in $\bT^d$ for all irrational point $y\in\bT^d$.
\end{theorem}

\begin{remark}\label{Berendfisubgp}Notice that if $\alpha:\bZ^r\curvearrowright\bT^d$ is as in Theorem \ref{Berend} and $H\leq\bZ^r$ is a finite-index subgroup, then $H\cong\bZ^r$ and $\alpha|_H$ satisfies the same conditions in Theorem \ref{Berend} as well. 
\end{remark}

In particular, Berend's theorem covers the special case of Cartan actions which contain totally irreducible elements\footnote {Note that according to our definition, a Cartan action always has irreducible elements, but not necessarily totally irreducible elements.} (see Lemma \ref{Cartanspecial}), where a Cartan action is defined by the following maximality condition:

\begin{definition}\label{Cartan}A faithful $\bZ^r$-action $\alpha$ on $\bT^d$ by automorphisms is {\bf Cartan} if the abelian subgroup $\alpha(\bZ^r)\leq\GL(d,\bZ)$ contains an irreducible element and is maximal in rank, i.e. there exists no intermediate abelian subgroup $G$ such that $\alpha(\bZ^r)\leq G\leq \GL(d,\bZ)$ and $\rank(G)>r$.\end{definition}

Here and throughout the rest of paper, $\rank(G)$ always refers to the torsion-free rank of a finitely generated abelian group $G$.

We remark that our definition of Cartan action here is more general than what was adopted by some previous authors, as in our case a Cartan action is not necessarily $\bR$-split, i.e. some of the common eigenspaces of $\alpha(\bZ^r)$ may be defined only over $\bC$ but not over $\bR$.

Cartan actions are of particular number-theoretical interest since a Cartan action is, up to finite index, conjugate to the multiplicative action of $U_K$, the group of units of some number field $K$ of degree $d$, on an arithmetic compact quotient of $K\otimes_\mathbb Q\mathbb R$ (see Proposition \ref{Gfield}).

\subsection{Main results}

In this paper, we try to understand what happens if the action is no longer irreducible. More precisely, for a $\bZ^r$-Cartan action $\alpha$ on $\bT^d$ as in Theorem \ref{Berend}, we consider the diagonal $\bZ^r$-action $\alpha_\triangle$ on $(\bT^d)^2$ given by \begin{equation}\label{diagaction}\alpha_\triangle^\bfn.y=(\alpha^\bfn.y^{(1)},\alpha^\bfn.y^{(2)}), \forall\bfn\in\bZ^r,\forall y=(y^{(1)},y^{(2)})\in(\bT^d)^2.\end{equation}

Recall the notion of topological transitivity:

\begin{definition}\label{topotrans} For a continuous action $\rho$ by a group $G$ on a compact metric space $\Omega$, a $\rho$-invariant closed subset $A\subset\Omega$ is {\bf topologically transitive} if for any two non-empty subsets $U,V\subset A$ that are both relatively open inside $A$, there is $g\in G$ such that $(\rho(g).U)\cap V\neq\emptyset$. 

It is known that $A$ is topologically transitive if and only if it is the orbit closure $\overline{\{\rho(g).\omega: g\in G\}}$ of some $\omega\in\Omega$, see \cite{GH55}*{Theorem 9.22}.\end{definition}

One hopes to classify all closed subsets of $(\bT^d)^2$ that are invariant and topologically transitive under the diagonal action $\alpha_\triangle$. A few candidate types are listed below.

{\noindent\bf($0$-dimensional)} Clearly if $x\in(\bT^d)^2$ is rational then the orbit $\{\alpha_\triangle^\bfn.x:\bfn\in\bZ^r\}$ is a finite set of rational points.

{\noindent\bf($2d$-dimensional)} It is not hard to see that $\alpha_\triangle$ is ergodic with respect to the Lebesgue measure on $(\bT^d)^2$, hence the orbit closure of almost every point is $(\bT^d)^2$ itself.

{\noindent\bf($d$-dimensional)} There is an intermediate category of $d$-dimensional sets. For example, let $y=(y^{(1)},0)$ where $y^{(1)}\in\bT^d$ is irrational. Then by Theorem \ref{Berend}, $\{\alpha^\bfn.y^{(1)}:\bfn\in\bZ^r\}$ is dense in $\bT^d$. So the orbit closure of $y$ under $\alpha_\triangle$ is $\bT^d\times\{0\}$, and thus $\bT^d\times\{0\}$ is a topologically transitive $\alpha_\triangle$-invariant subtorus. More generally, there are many $d$-dimensional subtori of $(\bT^d)^2$ that are invariant under the action $\alpha_\triangle$. In general, if $A_0$ is such an invariant subtorus and $A_1\subset (\bT^d)^2$ is the translate of $A_0$ by a rational point, then the stabilizer $H=\{\bfn:\alpha_\triangle^\bfn.A_1=A_1\}$ is a finite-index subgroup of $\bZ^r$. Choose representatives $\bfn_1,\cdots,\bfn_s$ such that $\bZ^r/H=\{\bfn_1+H,\cdots,\bfn_s+H\}$, and denote $A_t=\alpha_\triangle^{\bfn_t}.A_1$, for $t=1,\cdots,s$. Then the disjoint union $\bigsqcup_{t=1}^sA_t$ is $\alpha_\triangle$-invariant and topologically transitive.

All three types of invariant sets above are said to be {\it homogeneous}. One may speculate that these exhaust all topologically transitive $\alpha_\triangle$-invariant closed subsets of $(\bT^d)^2$. It turns out this is true when $r\geq 3$ but fails in rank $2$.

\begin{theorem}\label{Cartanjoining}Let $r \geq 2$ and $\alpha$ be a faithful $\bZ^r$-Cartan action on $\bT^d$ by automorphisms such that $\alpha^\bfn$ is a totally irreducible toral automorphism for at least one $\bfn \in\bZ^r$. Let $\alpha_\triangle:\bZ^r\curvearrowright\bT^d$ be the diagonal action in~(\ref{diagaction}).
\begin{enumerate}[label=\textup{(\arabic*)},leftmargin=*]
\item When $r \geq 3$, if an infinite proper closed subset $A$ of $(\bT^d)^2$ is invariant and topologically transitive under $\alpha_\triangle$, then there is an $\alpha_\triangle$-invariant $d$-dimensional subtorus $A_0 \subset \bT^{2d}$ such that $A$ decomposes into a finite disjoint union $\bigsqcup_{t=1}^s A_t$, where each $A_t$ is a translate of~$A_0$. Any $\alpha _ \triangle$-invariant closed subset of $(\bT^d)^2$ is a union of finitely many topologically transitive closed and invariant sets.

\item If $r=2$, then there is a point $\bfy \in (\bT^d)^2$ whose orbit closure is an infinite proper subset of $(\bT^d)^2$ but not homogeneous. Actually there are three $d$-dimensional subtori $T_1,T_2,T_3 \subset(\bT^d)^2$ transverse to each other, such that the orbit closure is a disjoint union:
\begin{equation}\overline{\{ \alpha_\triangle^\bfn.\bfy:\bfn \in\bZ^2 \}}=\{ \alpha_\triangle^\bfn.\bfy:\bfn \in\bZ^2 \} \sqcup\big(\bigcup_{i=1}^3T_i\big).\end{equation}
\end{enumerate}
\end{theorem}

Theorem \ref{Cartanjoining} is the main result of this paper. It actually classifies all topologically transitive self-joinings of the $\bZ^r$-action $\alpha$ when $r\geq 3$. Recall that a {\it topological joining} between two group actions $\rho_k:G\curvearrowright\Omega_k, k=1,2$ is a subset $A\subset \Omega_1\times\Omega_2$ which is invariant under the product action $\rho_1\times\rho_2$ such that $\pi_k(A)=\Omega_k$ where $\pi_k$ is the projection to $\Omega_k$. Here $\rho_1\times\rho_2$ is defined by $(\rho_1\times\rho_2)(g).(\omega_1,\omega_2)=\big(\rho_1(g).\omega_1,\rho_2(g).\omega_2\big)$ (see \cite{G03}).

We note that the analogous question in the measure theoretic category was studied by Kalinin and Katok \cite{KK02}. Examples such as Maucourant's \cite{M10} show that the measure theoretic and topological questions behave quite differently in this context.

One motivation behind this paper is the similarity between the actions investigated here and actions by higher-dimensional diagonal groups on homogeneous spaces. A general conjecture was raised by Margulis \cite{M00} on the rigidity of orbit closures under those diagonal actions. Recently several counterexamples to Margulis' Conjecture were constructed in different settings by the works of Maucourant \cite{M10}, Shapira \cite{S11}, Lindenstrauss and Shapira \cite{LS11} and Tomanov \cite{T10}. Our example of irregular orbit closures in the rank 2 may be considered as another manifestation of the same phenomenon.

Under certain circumstances it is possible to extend our results to semigroup actions. However, there is some delicateness in such an extension which is already evident in the case of $d = 1$.
In a follow-up paper we prove that if $a,b,c$ are pairwise relatively prime,  all topological self-joinings of the action of the semigroup generated by $a, b, c$ on $\bT$ are homogeneous. However, it is easy to construct counterexamples if the greatest common divisor of $(a,b,c)\neq 1$, e.g. if all are even.

Besides $\alpha_\triangle$-orbits of generic points, our techniques  also allow us to analyze orbits of rational points. Recall a subset of a metric space is said to be {\it $\epsilon$-dense} if it intersects every open ball of radius $\epsilon$.

\begin{theorem}\label{nondenserational}Suppose $r\geq 3$ and $\alpha$ is a faithful $\bZ^r$-Cartan action on $\bT^d$ by automorphisms such that $\alpha^\bfn$ is a totally irreducible toral automorphism for at least one $\bfn\in\bZ^r$. Then for all $\epsilon>0$, there exist a finite number of subsets $A_1,\cdots,A_s\subset(\bT^d)^2$, each of which is a translate of a (possibly different) $d$-dimensional $\alpha_\triangle$-invariant subtorus in $(\bT^d)^2$, such that $\bigcup_{t=1}^sA_t$ covers all rational points in $(\bT^d)^2$ whose $\alpha_\triangle$-orbits are not $\epsilon$-dense in $(\bT^d)^2$.\end{theorem}

\subsection{Organization of the paper}

Section \ref{numbermodel} gives some basic facts about Cartan actions by toral automorphisms. In particular, it will be explained that the action $\alpha:\bZ^r\curvearrowright\bT^d$ is equivalent to an algebraic $\bZ^r$-action $\zeta$ on a twisted torus $X$ which arises from some number field $K$, and in the rest of the paper, we will study the diagonal action $\zeta_\triangle:\bZ^r\curvearrowright X^2$ rather than $\alpha_\triangle$. 
In Section \ref{homogeneous}, by analysing characters of $X$ we describe the homogeneous $\zeta_\triangle$-invariant subsets in $X^2$. As mentioned above, there are $0$-dimensional, $d$-dimensional and $2d$-dimensional ones, the most interesting of which are the $d$-dimensional ones; we show how these can be parametrized in terms of the number field $K$.

In Section \ref{containhomo}, it is shown that any infinite $\zeta_\triangle$-invariant closed subset contains a $d$-dimensional homogeneous invariant set; this holds even when $r=2$.

Section \ref{rank3} establishes the rigidity of the diagonal action $\zeta_\triangle$ when the rank is at least $3$. The strategy is that, once a topologically transitive infinite $\zeta_\triangle$-invariant closed subset $A$ is known to contain a $d$-dimensional homogeneous invariant set as a proper subset, we can repeatedly add new $d$-dimensional homogeneous invariant subsets to it. But the union of infinitely many $d$-dimensional homogeneous invariant sets would become dense in $X^2$ so $A$ must be $X^2$. Two main ingredients of this proof are respectively a controlled recurrence argument (Lemma \ref{selfreturn}) that moves a point around while keeping it away from known homogeneous subsets; and an extension of Berend's Theorem to rigidity properties of non-hyperbolic abelian actions by toral automorphisms (Proposition \ref{nonhyp}, which is a special case of a more general fact proved in \cite{W10subaction}). Using the same techniques, we also provide a proof of Theorem \ref{nondenserational} in Section \ref{rank3}.

In Section \ref{rank2}, we construct a non-homogeneous orbit closure in the $r=2$ case. It is well known that Theorem \ref{Berend} fails for $r=1$ in which case, for example, homoclinic points have non-homogeneous orbit closures. In \cite{M10}, Maucourant gave examples of higher rank algebraic diagonal actions which have orbits wandering back and forth between several homogeneous submanifolds. Our counterexample bears resemblance to these examples. We construct a point $\bfx=(x^{(1)},x^{(2)})\in X^2$ such that $x^{(1)}$, $x^{(2)}$ and $x^{(1)}+x^{(2)}$ lie respectively in one of three foliations through the origin of $X$ along different common eigenspaces of the group action. So as elements of the group action all have determinant $1$, when a large element from the action is applied, at least one of these three expressions is attracted towards $0$.

\section{A number-theoretical model of the group action}\label{numbermodel}

For future convenience, we translate the problem into a number-theoretical setting.

\subsection{Cartan actions and groups of units in number fields}
Let $K$ be a degree $d$ number field with $r_1$ real embeddings and $r_2$ conjugate pairs of complex embeddings, where $r_1+2r_2=d$. Denote the real embeddings by $\sigma_1,\cdots,\sigma_{r_1}$ and the conjugate pairs of complex ones by $(\sigma_{r_1+1},\sigma_{r_1+r_2+1})$, $(\sigma_{r_1+2},\sigma_{r_1+r_2+2})$, $\cdots$, $(\sigma_{r_1+r_2},\sigma_{r_1+2r_2})$. Let $U_K$ be the group of units, which is of rank $r_1+r_2-1$. Define $\sigma:K\mapsto\bR^{r_1}\oplus\bC^{r_2}$ by
\begin{equation}\label{fieldemb}\sigma:\theta\mapsto\big(\sigma_1(\theta),\cdots,\sigma_{r_1}(\theta),\sigma_{r_1+1}(\theta),\cdots,\sigma_{r_1+r_2}(\theta)\big).\end{equation}
then $\sigma$ is an embedding between additive groups and induces an isomorphism $K\otimes_\bQ\bR\cong\bR^{r_1}\oplus\bC^{r_2}$.

Write $\theta.(\mu\otimes x)=\theta\mu\otimes x, \forall \theta,\mu\in K, \forall x\in\bR$; which gives a multiplicative action by $K$ on $K\otimes_\bQ\bR$. If we identify $K\otimes_\bQ\bR$ with $\bR^{r_1}\oplus\bC^{r_2}$ via $\sigma$ then this multiplicative action can be equivalently defined by:
\begin{equation}\label{fieldmulti}\theta.(\tilde x_1,\cdots,\tilde x_{r_1+r_2})=\big(\sigma_1(\theta)\tilde x_1,\cdots,\sigma_{r_1+r_2}(\theta)\tilde x_{r_1+r_2}\big),\end{equation}
which is compatible with $\sigma$ in the sense that \begin{equation}\label{embmulti}\theta.\sigma(\mu)=\sigma(\theta\mu), \forall \theta,\mu\in K.\end{equation}

\begin{definition}\label{CMfield}A number field $K$ is {\bf CM} if it is a totally complex quadratic extension of a totally real number field.\end{definition}

The following proposition is a special case of a well-known correspondence. For general statements and proofs, we refer the reader to \cite{S95}*{\S 7 \& \S 29} and \cite{EL04}, as well as \cite{EL03}*{Prop. 2.1}. See also for example \cite{W10effective}*{Thm. 2.12} for a self-contained proof in this special case.

\begin{proposition}\label{Gfield}Suppose $\alpha:\bZ^r\curvearrowright\bT^d$ is a faithful $\bZ^r$-Cartan action by automorphisms such that $\alpha(\bZ^r)$ contains a totally irreducible toral automorphism, then there exist \begin{itemize}[leftmargin=*]
\item a common eigenbasis $\{w_1,\cdots,w_d\}$ in $\bC^d$ with respect to which each $\alpha^\bfn$ can be diagonalized as $\diag{\zeta_1^\bfn,\zeta_2^\bfn,\cdots,\zeta_d^\bfn}$ where $\zeta_1^\bfn,\cdots,\zeta_{r_1}^\bfn\in\bR$; $\zeta_{r_1+1}^\bfn,\cdots,\zeta_d^\bfn\in\bC$ and $\zeta_{r_1+j}^\bfn=\overline{\zeta_{r_1+r_2+j}^\bfn}$ for $j=1,\cdots,r_2$;
\item a non-CM number field $K$ of degree $d$ with $r_1$ real embeddings $\sigma_1\cdots,\sigma_{r_1}$ and $r_2$ conjugate pairs of complex ones $(\sigma_{r_1+1}, \sigma_{r_1+r_2+1})$, $(\sigma_{r_1+2}, \sigma_{r_1+r_2+2})$, $\cdots$, $(\sigma_{r_1+r_2}, \sigma_d)$ where $r_1+r_2=r+1$, $r_1+2r_2=d$;
\item a group embedding $\zeta:\bfn\mapsto\zeta^\bfn$ of $\bZ^r$ into the group of units $U_K$;
\item a subgroup $\Gamma\leq\sigma(K)$ which is a cocompact lattice in $\bR^{r_1}\oplus\bC^{r_2}$, where $\sigma$ is given in (\ref{fieldemb}); moreover, $\Gamma$ is commensurable with the image under $\sigma$ of the ring of integers $\mathcal{O}_K$ of $K$,
\end{itemize}
such that:
\begin{itemize}[leftmargin=*]
\item $\zeta_i^\bfn=\sigma_i(\zeta^\bfn)$ for all $i\in\{1,\cdots,d\}$ and $\bfn\in \bZ^r$;
\item the image $\zeta(\bZ^r)=\{\zeta^\bfn:\bfn\in\bZ^r\}$ has finite index in $U_K$;
\item for all $\bfn\in \bZ^r$, the multiplication (\ref{fieldmulti}) by $\zeta^\bfn$ on $\bR^{r_1}\oplus\bC^{r_2}$ preserves $\Gamma$, and therefore induces a multiplicative $\bZ^r$-action $\zeta$ on $(\bR^{r_1}\oplus\bC^{r_2})/\Gamma$: $$\zeta^\bfn.(\tilde x\mathrm{\ mod\ }\Gamma)=(\zeta^\bfn.\tilde x\mathrm{\ mod\ }\Gamma), \forall\tilde x\in \bR^{r_1}\oplus\bC^{r_2}, \forall \bfn\in\bZ^r;$$
\item $\alpha$ is algebraically conjugate to the multiplicative action $\zeta:\bZ^r\curvearrowright(\bR^{r_1}\oplus\bC^{r_2})/\Gamma$, i.e.,  there is a continuous group isomorphism $\psi:\bT^d\overset\sim\to(\bR^{r_1}\oplus\bC^{r_2})/\Gamma$ such that $\alpha^\bfn.y=\psi^{-1}(\zeta^\bfn.\psi(y))$ for all $\bfn\in\bZ^r$ and $y\in\bT^d$.
\end{itemize}\end{proposition}

\begin{remark}\label{totalirrrmk}As $\zeta^\bfn$ is the abstract eigenvalue of the integer matrix $\alpha^\bfn$, $\alpha^\bfn$ is irreducible if and only if $\zeta^\bfn$ is an algebraic number of degree $d$, or in other words, $\bQ(\zeta^\bfn)=K$. Hence $\alpha^\bfn$ is totally irreducible if and only if for all non-zero integer $l$, $\zeta^{l\bfn}$ is not contained in any proper subfield of $K$.\end{remark}

In particular, Proposition \ref{Gfield} implies Cartan actions with totally irreducible elements are indeed covered by Theorem \ref{Berend} as a special case.

\begin{lemma}\label{Cartanspecial}Suppose $r\geq 2$, $\alpha$ is a faithful $\bZ^r$-Cartan action on $\bT^d$ by automorphisms and $\alpha(\bZ^r)$ contains a totally irreducible element, then $\alpha$ satisfies the conditions in Theorem \ref{Berend}.\end{lemma}

\begin{proof}Conditions (1) and (2) in Theorem \ref{Berend} are already assumed. It suffices to check (3). By Proposition \ref{Gfield}, for the $i$-th common eigenvector $v_i$, $\frac{|\alpha^\bfn.v_i|}{|v_i|}=|\zeta_i^\bfn|$. By Dirichlet's Unit Theorem, for any index $i$, there is $u\in U_K$ whose $i$-th embedding $\sigma_i(u)$ has norm greater than $1$. As $\zeta(\bZ^r)$ is of finite index in $U_K$, there is a positive integer $q$ such that $u^q\in \zeta(\bZ^r)$. In other words, $\exists\bfn\in\bZ^r$, $\zeta^\bfn=u^q$; hence  $|\zeta_i^\bfn|=|\sigma_i(\zeta^\bfn)|=|\sigma_i(u^q)|=|\sigma_i(u)|^q>1$. This verifies condition (3) in Theorem \ref{Berend}. \end{proof}

\subsection{Notations}\label{Xnota} Throughout this paper let $\alpha$ be a $\bZ^r$-Cartan action on $\bT^d$ by automorphisms such that $\alpha(\bZ^r)$ contains a totally irreducible element.

Let $K$, $\Gamma$, $\psi$, $\zeta$ and $\sigma_i$'s be as in Proposition \ref{Gfield}. 
Write $I=\{1,\cdots,r_1+r_2\}$ and to each $i\in I$ associate a subspace $V_i$ of $\bR^{r_1}\oplus\bC^{r_2}$ as follows: $V_i$ is the $i$-th copy of $\bR$ in $\bR^{r_1}\oplus\bC^{r_2}$ if $1\leq i\leq r_1$; and for $r_1< i\leq r_1+r_2$, $V_i$ is the $(i-r_1)$-th copy of $\bC$, so that $\bR^{r_1}\oplus\bC^{r_2}=\bigoplus_{i\in I}V_i$. Set $\bF_i=\bR$ or $\bC$ depending on whether $i \leq r_1$ or~$>r_1$, and let $d_i=\dim_\bR \bF_i=\dim_\bR V_i$.

Denote \begin{equation}X=(\bR^{r_1}\oplus\bC^{r_2})/\Gamma=(\bigoplus_{i\in I}V_i)/\Gamma,\end{equation} and let $\pi$ be the canonical projection from  $\bR^{r_1}\oplus\bC^{r_2}$ to $X$.
For $x\in X$, let \[\|x\|=\min_{\substack{\tilde x\in \bR^{r_1}\oplus\bC^{r_2}\\\pi(\tilde x)=x}}|\tilde x|.\] Then $(x,x')\mapsto\|x-x'\|$ is a distance on $x$ and this makes $X$ a locally Euclidean metric space.

In order to characterize how the action $\zeta$ expands or contracts different $V_i$'s, for all $i\in I$ we introduce a group morphism $\lambda_i:\bZ^r\mapsto\bR$ by  \begin{equation}\lambda_i(\bfn)=\log|\zeta_i^\bfn|,\end{equation} and construct a map $\cL:\bZ^r\mapsto\bR^I$ by \begin{equation}\label{Logmap}\cL(\bfn)=\big(\lambda_i(\bfn)\big)_{i\in I}.\end{equation} Then $\cL$ is a group morphism and its image $\cL(\bZ^r)$ is a finite-index subgroup in the set $\big\{\big(\log|\sigma_i(u)|\big)_{i\in I}:u\in U_K\big\}$.

\begin{remark}\label{dirichlet}By Dirichlet's Unit Theorem, the set $\big\{\big(\log|\sigma_i(u)|\big)_{i\in I}:u\in U_K\big\}$ is a full-rank lattice in the subspace \begin{equation}\label{Wdef}W=\{w\in\bR^I:\sum_{i\in I}d_iw_i=0\}.\end{equation} Since $\{\zeta^\bfn:\bfn\in\bZ^r\}$ has finite index in $U_K$ by Proposition \ref{Gfield}, for any finite index subgroup $H$ in $\bZ^r$, $\cL(H)$ is a finite-index subgroup in the above lattice, and thus is itself a full-rank lattice in $W$.\end{remark}

Write \begin{equation}X^2=X\oplus X=(\bR^{r_1}\oplus\bC^{r_2})^2/\Gamma^2.\end{equation} Denote by $\pi_\triangle$ the quotient map from $(\bR ^ {r _ 1} \oplus \bC ^ {r _ 2}) ^2$ to $X ^2$ and equip $X^2$ with a distance given by $\|\bfx \|=\min_{\substack{\tilde \bfx \in (\bR^{r_1}\oplus \bC^{r_2})^2\\\pi_\triangle(\tilde \bfx)=\bfx}}|\tilde \bfx|$.

$(\bR^{r_1}\oplus\bC^{r_2})^2\cong\bR^{2d}$ decomposes as $\bigoplus_{i\in I, k=1,2}V_i^{(k)}$ where $V_i^{(k)}$ is the $i$-th component $V_i$ in the $k$-th copy of $\bR^{r_1}\oplus\bC^{r_2}$. Write \begin{equation}\label{Visquare}V_i^\square=V_i^{(1)}\oplus V_i^{(2)},\end{equation} then \begin{equation}\label{Visquaresum}(\bR^{r_1}\oplus\bC^{r_2})^2=\bigoplus_{i\in I} V_i^\square.\end{equation}
Set \begin{equation}\label{X2multi}\zeta_\triangle^\bfn.(x^{(1)},x^{(2)})=(\zeta^\bfn.x^{(1)},\zeta^\bfn.x^{(2)}),\forall x^{(1)},x^{(2)}\in X,\forall\bfn\in\bZ^r,\end{equation} and  \begin{equation}\label{field2multi}\zeta_\triangle^\bfn.(\tilde x^{(1)},\tilde x^{(2)})=(\zeta^\bfn.\tilde x^{(1)},\zeta^\bfn.\tilde x^{(2)}),\forall \tilde x^{(1)},\tilde x^{(2)}\in \bR^{r_1}\oplus\bC^{r_2},\end{equation} These make $\zeta_\triangle$ is a $\bZ^r$-action both on $X^2$ and on $(\bR^{r_1}\oplus\bC^{r_2})^2$; moreover, (\ref{field2multi}) projects to (\ref{X2multi}) via the quotient map $\pi_\triangle$. It follows from Proposition \ref{Gfield} that the actions $\alpha_\triangle:\bZ^r\curvearrowright(\bT^d)^2$ and $\zeta_\triangle:\bZ^r\curvearrowright X^2$ are conjuagte via the isomorphism $\psi_\triangle:(\bT^d)^2\overset\sim\to X^2$ given by \begin{equation}\psi_\triangle(y^{(1)},y^{(2)})=\big(\psi(y^{(1)}),\psi(y^{(2)})\big).\end{equation}

Notice $V_i^\square$ is isomorphic to either $\bR^2$ or $\bC^2$ and under this identification we have \begin{equation}\label{Visquaremulti}\zeta_\triangle^\bfn.\bfv=\zeta_i^\bfn \bfv,\qquad \forall\bfv\in V_i^\square,\forall\bfn\in\bZ^r,\end{equation} hence \begin{equation}\label{X2Lip}\|\zeta_\triangle^\bfn.\bfx\|\leq e^{\max_{i\in I}\lambda_i(\bfn)}\|\bfx\|,\qquad\forall \bfx\in X^2,\forall\bfn\in\bZ^r.\end{equation}

For a subset $A\subset X^2$, define its stabilizer by \begin{equation}\label{stab}\Stab_{\zeta_\triangle}(A)=\{\bfn\in\bZ^r:\zeta_\triangle^\bfn.A\subset A\}.\end{equation}

In this paper, $B_\epsilon(\omega)$ and $\mathring B_\epsilon(\omega)$ will respectively denote the closed and open balls of radius $\epsilon$ around a point $\omega$ in a metric space.

\subsection{Rigidity results in $X$} We now translate the special case of Theorem \ref{Berend} regarding Cartan actions to our more number theoretic setting.

\begin{remark}\label{rationaltorsion}As $\psi$ is a group isomorphism, $y\in\bT^d$ is a rational point if and only if $\psi(y)$ is a torsion point in the compact abelian group $X$. In consequence, $\bfy\in(\bT^d)^2$ is rational if and only if $\psi_\triangle(\bfy)$ is a torsion point in $X^2$.\end{remark}

\begin{theorem}\label{Berend'}Let $X$ and $\zeta$ be as above, then the only infinite $\zeta$-invariant closed subset of $X$ is $X$ itself. In fact, for all $x\in X$ and finite index subgroup $H$ of $\bZ^r$, $\{\zeta^\bfn.x: \bfn\in H\}$ is dense in $X$ unless $x$ is a torsion point.\end{theorem}

\begin{proof}Proposition \ref{Gfield} established a conjugation between the $\bZ^r$-actions $\zeta$ and $\alpha$. Thus by Lemma \ref{Cartanspecial}, we may apply Theorem \ref{Berend} to $\zeta$, this proves the theorem for $H=\bZ^r$ where we used Remark \ref{rationaltorsion}. The statement extends to any finite index subgroup $H$ by Remark \ref{Berendfisubgp}.\end{proof}

When $r \geq 3$, for a generic $x$ it turns out that it is not necessary to apply the full action by $H$ to get a dense orbit in $X$, and this fact is quite important in our analysis. The next proposition is a special case of \cite{W10subaction}*{Theorem 1.7}.

\begin{proposition}\label{nonhyp} In the setting of Theorem \ref{Berend'}, if $r\geq 3$ then for all $i\in I$, $\epsilon>0$ and $x\in X$, the set $\{\zeta^\bfn.x:\bfn\in H,|\lambda_i(\bfn)|<\epsilon\}$ is dense in $X$ unless $x$ can be written as $x_0+v$ where $x_0$ is a torsion point in $X$ and $v\in V_i$.\end{proposition}
\begin{proof}
In order to verify the proposition is indeed a special case of Theorem 1.7 from that paper it suffice to check conditions (C1)-(C3) listed below. To begin with, remark for all $j\in I$, $\lambda_j:\bZ^r\mapsto\bR$ can be uniquely extended to a linear map from $\bR^r$ to $\bR$, still denoted by $\lambda_j$.\newline

\begin{itemize}
\item[(C1)]{\it Suppose $L_i\subset(\bR^r)^*$ is the subspace of linear maps spanned by $\lambda_i$, then $\lambda_j\notin L_i$ for $j\neq i$;}
\item[(C2)]{\it The ``rank'' of  $\{\zeta^\bfn.x:\bfn\in H,|\lambda_i(\bfn)|<\epsilon\}$ (which by definition is $r-1$) is $\geq 2$;}
\item[(C3)]{\it For all $\epsilon>0$, there exists $\bfn\in H$ such that $|\lambda_i(\bfn)|<\epsilon$ and $\zeta^\bfn\in K$ is a totally irreducible element, i.e. $\zeta^{l\bfn}$ doesn't belong to any proper subfield of $K$ for all non-zero integer $l$.}\newline
\end{itemize}

Of these three conditions, (C2)  is obvious, and (C1) is also easy: indeed, suppose (C1) fails then there is $j \in I \backslash \{ i \}$ such that $\lambda_j(\bfn)=c \lambda_i(\bfn),\forall\bfn$ for some constant $c$. Hence the image $\cL(\bZ^r)$ lies in $\{(w_i)_{i \in I}\in W:w_j=cw_i \}$, which is a proper linear subspace of $W$. This contradicts the fact that $\cL(\bZ^r)$ is a full-rank lattice in $W$. Establishing (C3), however, is slightly more delicate.
\medskip

\noindent{\bf Proof of (C3).} Let $P_i=\{\eta\in\bR^r:\lambda_i(\eta)=0\}$, then $\dim P_i=r-1$ because $\lambda_i\in(\bR^r)^*$ is non-zero by Remark \ref{dirichlet}.

For any proper subfield $F$ of $K$, let $H_F=\{\bfn\in\bZ^r:\zeta^\bfn\in U_F\}$, which is a subgroup in $\bZ^r\subset\bR^r$, and call by $P_F$ its linear span in $\bR^r$. As $K$ is not a CM-field and $F$ is a proper subfield, it follows that $\rank(U_F)<\rank(U_K)=r$ (see e.g. \cite{P75}). Because $\zeta$ is a group embedding and its image has finite index in $U_K$, $\zeta(H_F)=\zeta(\bZ^r)\cap H_F$ has finite index in $U_F$, and thus $\dim P_F=\rank(H_F)=\rank(U_F)<r$.

We hope to show $P_i\not\subset P_F$. As $\dim P_i=r-1\geq\dim P_F$, it suffices to show $P_i\neq P_F$. Suppose for the moment that $P_i=P_F$, then for all $\bfn\in H_F$, $\lambda_i(\bfn)=0$ or equivalently $\sigma_i(\zeta^\bfn)=\zeta_i^\bfn$ has absolute value 1. Because $\zeta(H_F)$ has finite index in $U_F$, it follows that $|\sigma_i(u)|=1$, $\forall u\in U_F$. Let $\tau=\sigma_i|_F$, then $\tau$ is an archimedean embedding of $F$ and the above property rewrites $|\tau(u)|=1$, $\forall u\in U_F$, which contradicts Dirichlet's Unit Theorem for $F$. Therefore $P_i\not\subset P_F$ and thus $P_i\cap P_F$ is a proper subspace in $P_i$.

In the rest of proof let $F$ run over all proper subfields $F$ of $K$, as there are only finitely many such, $\bigcup_F(P_i\cap P_F)\subsetneq P_i$, or equivalently, \begin{equation}P_i\not\subset\bigcup_FP_F.\end{equation}

Take an arbitrary non-zero vector $\eta_0\in P_i\backslash\big(\bigcup_FP_F\big)$. As $\bigcup_FP_F$ is closed, there is $\delta>0$ such that \begin{equation}\label{nonhyp1}B_\delta(\eta_0)\cap\big(\bigcup_FP_F\big)=\emptyset.\end{equation} In particular, $\delta<|\eta_0|$.

For any $\eta\in\bR^r$, let $\bar\eta$ denote its projection in the quotient $\bT^r=\bR^r/\bZ^r$.  By compactness of $\bT^r$, there is an increasing sequence of positive integers $m_1,m_2,\cdots$ such that $\{\overline{m_k\eta_0}\}_{k=1}^\infty$ converges. In particular, there exist $m,m'\in\bN$ such that \begin{equation}\label{nonhyp2}m-m'\geq\frac\epsilon{|\bZ^r/H|\cdot \|\lambda_i\|\delta}\end{equation} and the distance between $\overline{m\eta_0}$ and $\overline{m'\eta_0}$ is less than $\frac\epsilon{|\bZ^r/H|\cdot \|\lambda_i\|}$, where $\|\lambda_i\|$ denotes the norm of the linear functional $\lambda_i$. This means for some $\bfn_0\in\bZ^r$, \begin{equation}\label{nonhyp3}|(m-m')\eta_0-\bfn_0|<\frac\epsilon{|\bZ^r/H|\cdot \|\lambda_i\|}.\end{equation}

As $\delta<|\eta_0|$, by comparing (\ref{nonhyp2}) with (\ref{nonhyp3}) we see $\bfn_0\neq 0$.
Using pigeonhole principle one can easily check there is $p\in\{1,\cdots,|\bZ^r/H|\}$ such that $\bfn=p\bfn_0$ is a non-trivial element in $H$.

Because $\eta_0\in P_i$, $\lambda_i\big((m-m')\eta_0\big)=0$ and (\ref{nonhyp3}) implies \begin{equation}|\lambda_i(\bfn_0)|\leq\|\lambda_i\|\cdot \frac\epsilon{|\bZ^r/H|\cdot \|\lambda_i\|}=\frac\epsilon{|\bZ^r/H|}.\end{equation}
Therefore \begin{equation}|\lambda_i(\bfn)|=p|\lambda_i(\bfn_0)|\leq |\bZ^r/H|\cdot|\lambda_i(\bfn_0)|<\epsilon.\end{equation}

On the other hand, remark \begin{equation}\begin{split}\Big|\frac{\bfn_0}{m-m'}-\eta_0\Big|=&\frac{|\bfn_0-(m-m')\eta_0|}{m-m'}\\
<&{\frac\epsilon{|\bZ^r/H|\cdot \|\lambda_i\|}}\bigg/{\frac\epsilon{|\bZ^r/H|\cdot \|\lambda_i\|\delta}}
=\delta.
\end{split}\end{equation}
Thus by (\ref{nonhyp1}), $\frac{\bfn_0}{m-m'}\notin\bigcup_FP_F$. So for any $l\in\bZ\backslash\{0\}$, since $l\bfn=lp\bfn_0$ is proportional to $\frac{\bfn_0}{m-m'}$, it doesn't belong to $\bigcup_FP_F$; and  thus by construction of $P_F$, $\zeta^{l\bfn}$ is not in any proper subfield $F$ of $K$. This completes the verification of condition (C3) and finally establishes the proposition.\end{proof}

\subsection{Characters in the new setting}
We identify $\bR^{r_1}\oplus\bC^{r_2}$ with its dual, via the pairing \begin{equation}\label{fielddualformula}\langle\xi,\tilde x\rangle=\sum_{i=1}^{r_1}\xi_i\tilde x_i+\sum_{i=r_1+1}^{r_1+r_2}2\re(\xi_i\tilde x_i)\text{ mod }\bZ,\quad\forall\xi,\tilde x\in \bR^{r_1}\oplus\bC^{r_2}.\end{equation} 
The dual group to $X=(\bR^{r_1}\oplus\bC^{r_2})/\Gamma$ is the group \begin{equation}\label{Xdual}\widehat{X}=\{\xi\subset \bR^{r_1}\oplus\bC^{r_2}:\langle\xi,\gamma\rangle=0(\mathrm{mod\ }\bZ),\forall\gamma\in\Gamma\}\end{equation} where $\langle\xi,\pi(\tilde x)\rangle=\langle\xi,\tilde x\rangle$ if $x=\pi(\tilde x)$ for $\xi\in\hat X$ and $\tilde x\in\bR^{r_1}\oplus\bC^{r_2}$. 
Since $X\cong\bT^d$, $\hat X\cong\bZ^d$.

The dual group $\widehat{X^2}$ is just $(\hat X)^2$. For $\bfxi=(\xi^{(1)},\xi^{(2)})\in\widehat{X^2}$ and $\bfx=(x^{(1)},x^{(2)})\in X^2$, the pairing is given by \begin{equation}\label{X2dualformula}\langle\bfxi,\bfx\rangle=\langle\xi^{(1)},x^{(1)}\rangle+\langle\xi^{(2)},x^{(2)}\rangle.\end{equation}

The $\bZ^r$-actions $\zeta$ and $\zeta_\triangle$ respectively induce dual actions on $\hat X$ and on $\widehat{X^2}$ in natural ways. To understand the dual actions, notice $\zeta$ preserves the lattice $\Gamma$, and thus by (\ref{Xdual}) $\hat X\subset\bR^{r_1}\oplus\bC^{r_2}$ is invariant under $\zeta^\bfn$ for all $\bfn\in\bZ^r$. Furthermore, 
 \begin{equation}\langle\xi,\theta.\tilde x\rangle=\langle \theta.\xi,\tilde x\rangle,\forall \theta\in K,\qquad\forall \xi,\tilde x\in \bR^{r_1}\oplus\bC^{r_2},\end{equation} 
 where both multiplications by $\theta$ are given by (\ref{fieldmulti}), so the restriction of the $\bZ^r$-action $\zeta$ to $\hat X$ is dual to $\zeta:\bZ^r\curvearrowright X$ in the sense that \begin{equation}\langle\zeta^\bfn.\xi,x\rangle=\langle\xi,\zeta^\bfn.x\rangle,\qquad\forall\xi\in\hat X,\forall x\in X,\forall\bfn\in\bZ^r.\end{equation} In addition, $\widehat {X^2}\subset(\bR^{r_1}\oplus\bC^{r_2})^2$ is stable under $\zeta_\triangle$ and 
 \[\langle\zeta_\triangle^\bfn.\bfxi,\bfx\rangle=\langle\bfxi,\zeta_\triangle^\bfn.\bfx\rangle,\qquad \forall\bfxi\in\widehat{X^2},\forall \bfx\in X^2,\forall n\in\bZ^r.\]

Before finishing this section we show that every character of $X$ arises from an algebraic number in $K$.

\begin{lemma}\label{XdualK}With the identification (\ref{Xdual}), $\hat X$ is a subset of $\sigma(K)$ commensurable with $\sigma(\mathcal{O}_K)$; in particular, every $\xi\in\hat X$ can be written as $\xi=\sigma(\theta)$ for some $\theta\in K$.\end{lemma}

\begin{proof} Denote by $X_K$ the $d$-dimensional torus $(\bR^{r_1}\oplus\bC^{r_2})/\sigma(\cO_K)$ where $\cO_K$ is the ring of integers in $K$, then \begin{equation}\widehat{X_K}=\{\xi\in\bR^{r_1}\oplus\bC^{r_2}:\langle\xi,\sigma(\mu)\rangle=0 (\text{mod }\bZ),\forall\mu\in\cO_K\}.\end{equation} Note $\sigma(\cO_K)\subset\widehat{X_K}$ because for all $\theta,\mu\in\cO_K$, \begin{equation}\begin{split}&\sum_{i=1}^{r_1}\sigma_i(\theta)\sigma_i(\mu)+\sum_{j=1}^{r_2}2\re\big(\sigma_{r_1+j}(\theta)\sigma_{r_1+j}(\mu)\big)\\ 
=&\sum_{i=1}^{r_1}\sigma_i(\theta\mu)+\sum_{j=1}^{r_2}\big(\sigma_{r_1+j}(\theta\mu)+\sigma_{r_1+r_2+j}(\theta\mu)\big)\\
=&\sum_{i=1}^d\sigma_i(\theta\mu)=\Tr_{K/\bQ}(\theta\mu)\in\bZ.\end{split}\end{equation}

Because both $\widehat{X_K}$ and $\sigma(\cO_K)$ are isomorphic to $\bZ^d$, $\sigma(\cO_K)$ is a finite-index subgroup in $\widehat{X_K}$. Moreover, as $\Gamma$ is contained in $\sigma(K)$, it is commensurable with $\sigma(\cO_K)$. It follows that $\hat X$ and $\widehat{X_K}$ are commensurable lattices in $\bR^{r_1}\oplus\bC^{r_2}$, and thus $\hat X$ is commensurable to $\sigma(\cO_K)$ as well. This implies $\hat X$ lies in the $\bQ$-span of $\sigma(\cO_K)$, which is exactly $\sigma(K)$.\end{proof}

\section{Description of homogeneous invariant subsets}\label{homogeneous}

Three types of homogeneous $\alpha_\triangle$-invariant closed sets in $(\bT^d)^2$ were described in the discussion preceding Theorem \ref{Cartanjoining}. In this part, we are going to discuss their counterparts in $X^2$.

Under the map $\psi_\triangle$, finite sets of rational points correspond to finite sets of torsion points by Remark \ref{rationaltorsion}, and $(\bT^d)^2$ just becomes $X^2$.


It remains to describe $d$-dimensional homogeneous invariant subsets in $X^2$.

\subsection{Homogeneous invariant subtori in $X^2$}

$0$ and $X^2$ are respectively the only $0$-dimensional and $2d$-dimensional subtori in $X^2$ and they are clearly $\zeta_\triangle$-invariant. We now construct $d$-dimensional $\zeta_\triangle$-invariant subtori.

\begin{definition}\label{dhomotor}For any $\kappa\in K$, we define \begin{equation}(\widehat{X^2})^\kappa=\{(\xi^{(1)},\xi^{(2)})\in\widehat {X^2},\ \xi^{(1)}+\kappa.\xi^{(2)}=0\}<\widehat{X^2},\end{equation}
where the expression $\xi^{(1)}+\kappa.\xi^{(2)}$ is calculated as an element in $\bR^{r_1}\oplus\bC^{r_2}$.
For $\kappa=\infty$, define \begin{equation}(\widehat{X^2})^\infty=\{(\xi^{(1)},0):\xi^{(1)}\in\hat X\}.\end{equation}

A {\bf $d$-dimensional homogeneous $\zeta_\triangle$-invariant subtorus} is a subgroup in $X^2$ of the form
\begin{equation}T^\kappa=\big((\widehat{X^2})^\kappa\big)^\bot=\{\bfx\in X^2:\langle\bfxi,\bfx\rangle=0\mathrm{\ (mod\ }\bZ),\forall\bfxi\in(\widehat{X^2})^\kappa\}\end{equation} where $\kappa\in K\cup\{\infty\}$.
\end{definition}

\begin{remark}\label{switch12}By exchanging coordinates $x^{(1)}$ and $x^{(2)}$, $(\widehat{X^2})^\infty$ becomes $(\widehat{X^2})^0$ and $T^\infty$ becomes $T^0$.\end{remark}

Now we justify that the subgroup $T^\kappa$ is indeed a subtorus.

\begin{lemma}\label{dhomotorprop}$\forall\kappa\in K\cup\{\infty\}$, $T^\kappa$ is a $d$-dimensional subtorus in $X^2$. \end{lemma}
\begin{proof} $T^\kappa$ is a closed subgroup in $X^2$. As $T^\kappa=\big((\widehat{X^2})^\kappa\big)^\bot$, its  Pontryagin dual is the quotient $\widehat{X^2}/(\widehat{X^2})^\kappa$. In order to show $T^\kappa\cong\bT^d$ it suffices to show  \begin{equation}\label{dhomotorprop1}\widehat{X^2}/(\widehat{X^2})^\kappa\cong\bZ^d.\end{equation}

Recall $\hat X$ is a full-rank lattice in the vector space $\bR^{r_1}\oplus\bC^{r_2}$ and so is $\widehat{X^2}=(\hat X)^2$ in $(\bR^{r_1}\oplus\bC^{r_2})^2\cong\widehat{(\bR^{r_1}\oplus\bC^{r_2})^2}\cong\bR^{2d}$. When $\kappa=\infty$, (\ref{dhomotorprop1}) follows directly from construction of $(\widehat{X^2})^\infty$. Assume from now on $\kappa\in K$, then it is easy to verify that \begin{equation}\label{dhomotorprop2}F^\kappa=\{(\xi^{(1)},\xi^{(2)})\in(\bR^{r_1}\oplus\bC^{r_2})^2:\xi^{(1)}+\kappa.\xi^{(2)}=0\}\end{equation} is a $d$-dimensional linear subspace of $(\bR^{r_1}\oplus\bC^{r_2})^2$. Moreover $(\widehat{X^2})^\kappa=\widehat{X^2}\cap F^\kappa$, hence is a sublattice of rank at most $d$. Notice the subset $\{(-\kappa.\xi,\xi):\xi\in\hat X\}\subset F^\kappa$ is a lattice of rank $d$.

In order to show $(\widehat{X^2})^\kappa$ has rank $d$, it suffices to show that $\forall\xi\in\hat X$, $\exists q\in\bN$ such that $q(-\kappa.\xi,\xi)\in\widehat{X^2}$ or equivalently $q\kappa.\xi\in\hat X$, which is furthermore equivalent to
\begin{equation}\label{dhomotorprop3}\langle\xi,q\kappa.\gamma\rangle=\langle q\kappa.\xi,\gamma\rangle=0\mathrm{\ (mod\ }\bZ),\forall\gamma\in\Gamma\end{equation} by (\ref{Xdual}). Since $\xi$ is in $\hat X$ itself, (\ref{dhomotorprop3}) would follow if we could find $q\in\bN$ such that $q\kappa.\gamma\in\Gamma$ for all $\gamma\in\Gamma$. Take a $\bZ$-basis $\{\gamma^1,\cdots,\gamma^d\}$ of $\Gamma$, it suffices to find $q$ such that $q\kappa.\gamma^j\in\Gamma,\forall j=1,\cdots,d$. By Proposition \ref{Gfield} $\Gamma\subset\sigma(K)$, and thus $\kappa.\gamma^j\in\sigma(\kappa.K)=\sigma(K)$ by (\ref{embmulti}). Recall $\sigma(K)$ is a $\bQ$-vector space which contains $\Gamma$ as a lattice of full rank. Thus for each $j$, there is $q_j\in\bN$ such that $q_j\kappa.\gamma^j\in\Gamma$. Then $q=\mathrm{gcd}(q_1,\cdots,q_d)$ satisfies (\ref{dhomotorprop3}). This shows $(\widehat{X^2})^\kappa$ is of rank $d$.

Therefore $\rank\big(\widehat{X^2}/(\widehat{X^2})^\kappa\big)=2d-d=d$. In order to establish (\ref{dhomotorprop1}) it remains to prove $\widehat{X^2}/(\widehat{X^2})^\kappa$ is torsion-free. In fact, assume $\widehat{X^2}/(\widehat{X^2})^\kappa$ contains a torsion element, then there exist $\bfxi\in\widehat{X^2}$ and a non-zero integer $n$ such that $\bfxi\notin(\widehat{X^2})^\kappa$ but $n\bfxi\in (\widehat{X^2})^\kappa$. Because $(\widehat{X^2})^\kappa=\widehat{X^2}\cap F^\kappa$ , this is equivalent to $\bfxi\notin F^\kappa$, $n\bfxi\in F^\kappa$, which cannot be true since $ F^\kappa$ is a vector space. Hence we proved (\ref{dhomotorprop1}) and in consequence $T^\kappa$ is isomorphic to $\bT^d$ as a topological subgroup in $X^2$ for all $\kappa\in K\cup\{\infty\}$.\end{proof}

\begin{corollary}\label{finiteTkappainter}Suppose $\kappa,\kappa'\in K\cup\{\infty\}$, $\kappa\neq\kappa'$ and $\bfx,\bfx'\in X^2$, then $(\bfx+T^\kappa)\cap(\bfx'+T^{\kappa'})$ is a finite set.\end{corollary}

\begin{proof}Suppose the intersection is non-empty and let $\bfy$ be a point from it. Then for all $\bfy'\in(\bfx+T^\kappa)\cap(\bfx'+T^{\kappa'})$, $\bfy'-\bfy\in T^{\kappa}\cap T^{\kappa'}$. Hence it suffices to show $T^{\kappa}\cap T^{\kappa'}$ is finite.

Note because $T^{\kappa}\cap T^{\kappa'}$ is a closed subgroup of $X^2$, it is enough to prove the annihilator $(T^{\kappa}\cap T^{\kappa'})^\bot$ has finite index in $\widehat{X^2}$. Notice that this annihilator contains both $(\widehat{X^2})^\kappa$ and $(\widehat{X^2})^{\kappa'}$. It was proved in the proof of the previous lemma that $(\widehat{X^2})^\kappa,(\widehat{X^2})^{\kappa'}\cong\bZ^d$. Therefore it suffices to show $(\widehat{X^2})^\kappa\cap(\widehat{X^2})^{\kappa'}=\{\bfzero\}$, since then $(\widehat{X^2})^\kappa\oplus(\widehat{X^2})^{\kappa'}$ would be isomorphic to $\bZ^{2d}$ and so has to be a finite index subgroup in $\widehat{X^2}$; it follows that the group $(T^{\kappa}\cap T^{\kappa'})^\bot$ which contains  $(\widehat{X^2})^\kappa\oplus(\widehat{X^2})^{\kappa'}$ is also of finite index and this would establish the corollary.

Suppose $\bfxi=(\xi^{(1)},\xi^{(2)})\in (\widehat{X^2})^\kappa\cap(\widehat{X^2})^{\kappa'}$. Assume first $\kappa,\kappa'\in K$, then $\xi^{(1)}+\kappa.\xi^{(2)}=\xi^{(1)}+\kappa'.\xi^{(2)}=0$ and thus $(\kappa-\kappa').\xi^{(2)}=0$. But since $\kappa-\kappa'\neq 0$ this can be true only if $\xi^{(2)}=0$. And furthermore $\xi^{(1)}=-\kappa'.\xi^{(2)}=0$. Therefore $\bfxi=\bfzero$. Now assume one of $\kappa$ and $\kappa'$ is $\infty$, without loss of generality let $\kappa'=\infty$ and $\kappa\in K$. Then by definition of $(\widehat{X^2})^\infty$, $\xi^{(2)}=0$. Moreover $\xi^{(1)}=\xi^{(1)}+\kappa.\xi^{(2)}=0$, and again $\bfxi$ vanishes. \end{proof}

\begin{lemma}For all $\kappa\in K\cup\{\infty\}$, $T^\kappa$ is a $\zeta_\triangle$-invariant subset. \end{lemma}

\begin{proof}To obtain invariance it suffices to show $(\widehat{X^2})^\kappa$ is invariant under the dual action $\zeta_\triangle:\bZ^r\curvearrowright\widehat{X^2}$. By Remark \ref{switch12}, we may assume $\kappa\in K$. Suppose $\bfxi=(\xi^{(1)},\xi^{(2)})\in(\widehat{X^2})^\kappa$, then $\xi^{(1)}+\kappa.\xi^{(2)}=0$ and for any $\bfn\in\bZ^r$, \begin{equation}\zeta^\bfn.\xi^{(1)}+\kappa.\zeta^\bfn.\xi^{(2)}=\zeta^\bfn.(\xi^{(1)}+\kappa.\xi^{(2)})=0.\end{equation} Hence $\zeta_\triangle.\bfxi=(\zeta^\bfn.\xi^{(1)},\zeta^\bfn.\xi^{(2)})$ belongs to $\widehat{X^2}$ as well. \end{proof}

\subsection{Local structures of homogeneous invariant subtori} Now we discuss the structure of $T^\kappa$ in further detail.

\begin{definition}\label{Vkappa}For all $\kappa\in K$, set \begin{equation}V^\kappa=\{(\tilde x^{(1)},\tilde x^{(2)})\in(\bR^{r_1}\oplus\bC^{r_2})^2:\kappa.\tilde x^{(1)}-\tilde x^{(2)}=0\}.\end{equation} Furthermore, when $\kappa=\infty$, let \begin{equation}V^\infty=\{(\tilde x^{(1)},\tilde x^{(2)})\in(\bR^{r_1}\oplus\bC^{r_2})^2:\tilde x^{(1)}=0\}.\end{equation}
Denote $V_i^\kappa=V^\kappa\cap V_i^\square$, $\forall i\in I, \forall\kappa\in K\cup\{\infty\}$, where $V_i^\square$ is defined in~(\ref{Visquare}).\end{definition}

Clearly $V^\kappa$ is a $d$-dimensional linear subspace for all $\kappa\in K\cup\{\infty\}$. Moreover, $V^\kappa=\bigoplus_{i\in I}V_i^\kappa$ for all $\kappa\in K\cup\{\infty\}$. Observe \begin{equation}\label{Vkappacoord}V_i^\kappa=\{(v,\sigma_i(\kappa)v):v\in V_i\}, \forall\kappa\in K,\text{ and }V_i^\infty=\{(0,v):v\in V_i\},\end{equation} so $V_i^\kappa$ is isomorphic to $V_i$ and has real dimension $1$ or $2$ depending on whether $i\leq r_1$ or not.

Furthermore, both $V^\kappa$ and $V_i^\kappa$ are invariant under the multiplicative action $\zeta_\triangle$.

It turns out that $V^\kappa$ is the tangent space of the subtorus $T^\kappa$:

\begin{lemma}\label{Tkappatangent}$T^\kappa=\pi_\triangle(V^\kappa)\cong V^\kappa/(\Gamma^2\cap V^\kappa)$, $\forall\kappa\in K\cup\{\infty\}$.\end{lemma}
Recall $\pi_\triangle$ denotes the projection from $(\bR^{r_1}\oplus\bC^{r_2})^2$ to $X^2$.
\begin{proof}Consider $T^\kappa$ as a Lie subgroup in the abelian Lie group $X^2$. To get the lemma, it suffices to show that the Lie algebra of $T^\kappa$ is $V^\kappa$. Since $\dim T^\kappa=d=\dim V^\kappa$, we only need to prove $\pi_\triangle(\bfv)\in T^\kappa$ for all $v\in V^\kappa$. This is clear for $\kappa=\infty$ as $V^\infty=\{0\}\times(\bR^{r_1}\oplus\bC^{r_2})$ and $T^\infty=\{0\}\times X$. Now let $\kappa\in K$ and assume $\bfv=(v,\kappa.v)$ where $v\in \bR^{r_1}\oplus\bC^{r_2}$. Take an arbitrary $\bfxi\in(\widehat{X^2})^\kappa\subset(\bR^{r_1}\oplus\bC^{r_2})^2$, then $\langle\bfxi,\pi_\triangle(\bfv)\rangle=(\langle\bfxi,\bfv\rangle\text{ mod }\bZ)$ but $\langle\bfxi,\bfv\rangle=\langle\xi^{(1)},v\rangle+\langle\xi^{(2)},\kappa.v\rangle=\langle\xi^{(1)},v\rangle+\langle\kappa.\xi^{(2)},v\rangle=\langle\xi^{(1)}+\kappa.\xi^{(2)},v\rangle=\langle0,v\rangle=0$. Thus $\pi_\triangle(\bfv)\in\big((\widehat{X^2})^\kappa\big)^\bot=T^\kappa$. \end{proof}

The next two corollaries follow easily from the lemma.

\begin{corollary}\label{Vkappatransversal}If $\kappa,\kappa'\in K\cup\{\infty\}$ and $\kappa\neq\kappa'$ then $(\bR^{r_1}\oplus\bC^{r_2})^2=V^\kappa\oplus V^{\kappa'}$.\end{corollary}
\begin{proof}Since $\dim(\bR^{r_1}\oplus\bC^{r_2})^2=2d=\dim V^\kappa+\dim V^{\kappa'}$ it suffices to show $V^\kappa\cap V^{\kappa'}=\{\bfzero\}$. Suppose not, then $V^\kappa\cap V^{\kappa'}$ has positive dimension. In consequence $T^\kappa\cap T^{\kappa'}\supset\pi_\triangle(V^\kappa\cap V^{\kappa'})$ has positive dimension and hence is infinite, which contradicts Corollary \ref{finiteTkappainter} by taking $\bfx=\bfx'=\bfzero$.\end{proof}

\begin{corollary}\label{kappacapinfty}If $\kappa\in K$, then $\forall\bfx\in X^2$, $(\bfx+T^\kappa)\cap T^\infty\neq\emptyset$.\end{corollary}
\begin{proof}It is equivalent to show $T^\kappa\cap(-\bfx+T^\infty)\neq\emptyset$. Suppose $\bfx=(x^{(1)},x^{(2)})\in X^2$ and choose $\tilde x^{(1)}\in \bR^{r_1}\oplus\bC^{r_2}$ which projects to $x^{(1)}$. Let $\bfx'=\big(x^{(1)},\pi(\kappa.\tilde x^{(1)})\big)\in X^2$. Then $-\bfx'=\pi_\triangle\big((-\tilde x^{(1)},-\kappa.\tilde x^{(1)})\big)\in\pi(V^\kappa)=T^\kappa$ and $(-\bfx')-(-\bfx)=\big(0,x^{(2)}-\pi(\kappa.\tilde x^{(1)})\big)\in T^\infty$. Therefore $-\bfx\in T^\kappa\cap(-\bfx+T^\infty)$.\end{proof}

Now we claim that Berend's theorem applies to the torus $T^\kappa$.

\begin{lemma}\label{TkappaBerend}Let $H\leq\bZ^r$ be a finite-index subgroup and $\kappa\in K\cup\{\infty\}$. Then the action of $H$ on $T^\kappa$ by $\zeta_\triangle$ satisfies the conditions in Theorem \ref{Berend}. Any $\zeta_\triangle$-invariant closed subset of $T^\kappa$ is either $T^\kappa$ itself or a finite set of torsion points. In particular, for $\bfx\in T^\kappa$, the $H$-orbit $\{\zeta_\triangle^\bfn.\bfx:\bfn\in H\}$ of $x$ under $\zeta^\triangle$ is dense in $T^\kappa$ unless $\bfx$ is a torsion point; and $T^\kappa$ is topologically transitive under the action of $H$ by $\zeta_\triangle$.\end{lemma}

\begin{proof}We only need to prove the lemma for $\kappa\in K$ because of Remark~\ref{switch12}.

By (\ref{Vkappacoord}), $V^\kappa$ can be identified with $\bR^{r_1}\oplus\bC^{r_2}$ where $(v,\kappa.v)$ corresponds to $v\in \bR^{r_1}\oplus\bC^{r_2}$. Lemma  \ref{Tkappatangent} implies $T^\kappa\cong(\bR^{r_1}\oplus\bC^{r_2})/\Gamma'$ where $\Gamma'=\{v\in\bR^{r_1}\oplus\bC^{r_2}:(\kappa.v,v)\in\Gamma^2\}=\{\gamma\in\Gamma:\kappa.\gamma\in\Gamma\}\subset\Gamma$. As we know $T^\kappa\cong\bT^d$, by Lemma \ref{Tkappatangent} $\Gamma'$ is a full-rank sublattice of $\Gamma$; in particular $\Gamma'\subset\sigma(K)$. Because $\zeta_\triangle^\bfn$ sends $(\kappa.v,v)$ to $\big(\kappa.(\zeta^\bfn.v),\zeta^\bfn.v\big)$, when we make the identification above, the restriction of $\zeta_\triangle$ to the invariant hyperplane $V^\kappa$ is conjugate to the action $\zeta:\bZ^r\curvearrowright\bR^{r_1}\oplus\bC^{r_2}$ in (\ref{fieldmulti}). In consequence, $\Gamma'$ is a $\zeta$-invariant lattice in $\sigma(K)$.

$K$, $\zeta$ and $\Gamma'$ satisify the conditions listed in Proposition \ref{Gfield}. Therefore the action $\zeta:\bZ^r\curvearrowright(\bR^{r_1}\oplus\bC^{r_2})/\Gamma'$ is algebraically conjugate to a faithful $\bZ^r$-Cartan action by $d$-dimensional toral automorphisms that contains a totally irreducible element. Therefore by Lemma \ref{Cartanspecial} $\zeta_\triangle:\bZ^r\curvearrowright T^\kappa$ satisfies the conditions in Theorem \ref{Berend}, and by Remark \ref{Berendfisubgp} so does the restriction $\zeta_\triangle|_H$. By Theorem \ref{Berend} the only $\zeta_\triangle$-invariant infinite closed subset of $T^\kappa$ is $T^\kappa$ itself. The density and topological transitivity statements follow from this.\end{proof}

\begin{corollary}\label{Tkappanonhyp}For all finite-index subgroup $H$ in $\bZ^r$, $i\in I$, $\epsilon>0$, $\kappa\in K\cup\{\infty\}$, and $\bfx\in T^\kappa$, \begin{equation}\overline{\{\zeta_\triangle^\bfn.\bfx:\bfn\in H,\lambda_i(\bfn)\in(-\epsilon,\epsilon)\}}=T^\kappa\end{equation} unless $\bfx$ can be written as $\bfx_0+\bfv$ where $\bfx_0\in T^\kappa$ is of torsion and $\bfv\in V_i^\kappa$.\end{corollary}
\begin{proof}Following the proof of the previous lemma, we obtain the corollary by applying Proposition \ref{nonhyp} to the multiplicative action $\zeta:\bZ^r\curvearrowright(\bR^{r_1}\oplus\bC^{r_2})/\Gamma'$, which is conjugate to the restriction of $\zeta_\triangle$ to $T^\kappa$.\end{proof}

\subsection{Homogeneous invariant subsets in $X^2$}

\begin{definition}\label{dhomoset}A {\bf $d$-dimensional homogeneous $\zeta_\triangle$-invariant subset} in $X^2$ is a subset of the form \begin{equation}\label{dhomoform}L=\{\zeta_\triangle^\bfn.\bfz':\ \bfn\in\bZ^r,\bfz'\in \bfz+T^\kappa\},\end{equation} where $\bfz$ is a given torsion point in $X^2$ and $\kappa$ is a given element in $K\cup\{\infty\}$.
\end{definition}

Such a subset is said to be ``homogeneous'' because it is a disjoint union of several  parallel subtori.

\begin{proposition}\label{dhomoprop}Let $L\subset X^2$ be a $d$-dimensional homogeneous $\zeta_\triangle$-invariant subset of the form (\ref{dhomoform}), then:
\begin{enumerate}[label=\textup{(\arabic*)},leftmargin=*]
\item $L$ is a finite disjoint union $\bigsqcup_{t=1}^sL_t$ where each component $L_t$ is a translate of the invariant subtorus $T^\kappa$ by some torsion element of $X^2$, $\kappa$ being the same as in (\ref{dhomoform});

\item $\Stab_{\zeta_\triangle}(L_1)=\Stab_{\zeta_\triangle}(L_2)=\cdots=\Stab_{\zeta_\triangle}(L_d)$, which is a subgroup of index $s$ in $\bZ^r$;

\item $L$ is invariant and topologically transitive under the $\bZ^r$-action $\zeta_\triangle$. For any point $\bfx\in L$ that is not a torsion element of $X^2$, \[\overline{\{\zeta_\triangle^\bfn.\bfx:\bfn\in\bZ^r\}}=L.\]
\end{enumerate}
\end{proposition}

\begin{proof} (1) Since $T^\kappa$ is $\zeta_\triangle$-invariant, the stabilizer $H=\Stab_{\zeta_\triangle}(\bfz+T^\kappa)$ contains $H_0=\Stab_{\zeta_\triangle}(\bfz)$, which is a finite-index subgroup of $\bZ^r$. It follows that $H$ itself, which is obviously a semigroup, must be a subgroup of finite index.

Suppose the finite quotient group $\bZ^r/H$ consists of $s$ different cosets  $\bfn_1+H,\cdots,\bfn_s+H$. Then for any $t$ and all $\bfn\in\bfn_t+H$,  $\zeta_\triangle^\bfn.(\bfz+T^\kappa)=\zeta_\triangle^{\bfn_t}.\bfz+\zeta_\triangle^{\bfn_t}.T^\kappa=\zeta_\triangle^{\bfn_t}.\bfz+T^\kappa$ is the translate of $T_k$ by another torsion element $\zeta_\triangle^{\bfn_t}.\bfz$, which we denote by $L_t$. Moreover the $L_t$'s are all different, since if $\zeta_\triangle^{\bfn_t}.(\bfz+T^\kappa)$ and $\zeta_\triangle^{\bfn_{t'}}.(\bfz+T^\kappa)$ are equal, then $\zeta_\triangle^{\bfn_t-\bfn_{t'}}.(\bfz+T^\kappa)=(\bfz+T^\kappa)$ and thus $\bfn_t-\bfn_{t'}\in H$, which contradicts the assumption that $\bfn_t+H\neq\bfn_{t'}+H$.
So \begin{equation}L=\bigcup_{t=1}^s\bigcup_{\bfn\in\bfn_t+H}\big(\zeta_\triangle^\bfn.(\bfz+T^\kappa)\big)=\bigcup_{t=1}^s\bigcup_{\bfn\in\bfn_t+H}L_t=\bigsqcup_{t=1}^sL_t.\end{equation} This proves part (1).\\

(2) The coset $\zeta_\triangle^\bfn.L_{t'}=\zeta_\triangle^\bfn.\zeta_\triangle^{\bfn_{t'}}.(\bfz+T^\kappa)=\zeta_\triangle^{\bfn+\bfn_{t'}}.(\bfz+T^\kappa)$ is equal to $L_t$ if and only if $\bfn\in\bfn_t-\bfn_{t'}+H$. Hence the action $\zeta_\triangle^\bfn$ stabilizes $L_t$ if and only if $\bfn\in H$, i.e. for all $t=1,\cdots,s$ the stabilizer $\Stab_{\zeta_\triangle}(L_t)$ is exactly $H$, which is of index $s$ in $\bZ^r$.\\

(3) The invariance is clear, and since there certainly are non-torsion points in $L$ topological transitivity follows from the second part of~(3). As $L_1,\cdots,L_s$ are permuted transitively by the group action $\zeta_\triangle$, it suffices to show for any $L_t$ and any $\bfx \in L_t$ which is not a torsion element, $\overline{\{ \zeta_\triangle^\bfn.\bfx:\bfn \in H \}}=L_t$.

Recall that $L_t=\zeta^{\bfn_t}.\bfz+T^\kappa$, and that
\begin{equation*} H_0=\Stab_{\zeta_\triangle}(\bfz)=\Stab_{\zeta_\triangle}(\zeta^{\bfn_t}.\bfz)<H \end{equation*}
and has finite index in $\bZ^r$. Observe $\bfx-\zeta^{\bfn_t}.\bfz$ is not a torsion element since otherwise so is $\bfx$ (as $\zeta^{\bfn_t}.\bfz$ is a torsion point). By Lemma \ref{TkappaBerend}, $\overline{\{ \zeta_\triangle^\bfn.(\bfx-\zeta^{\bfn_t}.\bfz):\bfn \in H_0 \}}=T^\kappa$.Since for $\bfn \in H_0$, $\zeta_\triangle^\bfn.(\bfx-\zeta^{\bfn_t}.\bfz)=\zeta_\triangle^\bfn.\bfx-\zeta^{\bfn_t}.\bfz$, we conclude that \[\overline{\{ \zeta_\triangle^\bfn.\bfx:\bfn \in H_0 \}}=\zeta^{\bfn_t}.\bfz+T^\kappa=L_t.\] This completes the proof.
\end{proof}

Recall that a $d$-dimensional homogeneous $\zeta_\triangle$-invariant subset is determined (non-uniquely) by a torsion point $\bfz$ and a slope $\kappa\in K\cup\{\infty\}$.

\begin{lemma}\label{finitedhomo}For all $\epsilon>0$, there are only finitely many $d$-dimensional homogeneous $\zeta_\triangle$-invariant subsets that are not $\epsilon$-dense in $X^2$.\end{lemma}
\begin{proof}[Proof:]
{\bf Step 1.} We claim first that there are only finitely many $\kappa\in K$ such that $T^\kappa$ is not $\epsilon$-dense in $X^2$. Suppose the opposite, then there is a sequence of $\kappa_h\in K\cup\{\infty\}$, all different from each other, and points $\bfy_h\in X^2$ such that $T^{\kappa_h}\cap B_\epsilon(\bfy_h)=\emptyset,\forall h$. By passing to a subsequence we may assume that $\bfy_h$ is within distance $\frac\epsilon2$ from some point $\bfy\in X^2$ for all $h$. Then $B_{\frac\epsilon2}(\bfy)\subset B_\epsilon(\bfy_h)$ and therefore by assumption $T^{\kappa_h}\cap B_{\frac\epsilon2}(\bfy)=\emptyset,\forall h$. So the subsets $T^{\kappa_h}$ do not converge to $X^2$ in the Hausdorff metric as $h$ tends to $\infty$. Because $T^{\kappa_h}$ is a closed subgroup of $X^2$, by \cite{B83}*{Lemma 4.7} there exists a non-zero character $\bfxi\in\widehat{X^2}$ which lies in the annihilator of $T^{\kappa_h}$ for infinitely many $h$'s. In other words, there are infinitely many different slopes $\kappa_h$ such that $\bfxi\in(\widehat{X^2})^{\kappa_h}$. Without loss of generality assume all these $\kappa_h$'s are in $K$.

By Lemma \ref{XdualK}, $\bfxi=(\sigma(\theta^{(1)}),\sigma(\theta^{(2)}))$ for some $\theta^{(1)},\theta^{(2)}\in K$.  By definition of $(\widehat{X^2})^{\kappa_h}$, $\sigma(\theta^{(1)})+\kappa_h.\sigma(\theta^{(2)})=0$ for infinitely many $\kappa_h$'s, hence $\theta^{(1)}+\kappa_h\theta^{(2)}=0$. In order to make this happen for more than one $\kappa_h\in K$, $\theta^{(1)}$ and $\theta^{(2)}$ must be both $0$; hence $\bfxi=\bfzero$ and a contradiction is obtained. This shows the claim.\\

\noindent{\bf Step 2.} Notice that if $T^\kappa$ is $\epsilon$-dense, then so is any translate of it. Therefore given the claim in Step 1, to prove the lemma it suffices to show that for any given $\epsilon>0$ and $T^\kappa$, there are only finitely many  $d$-dimensional homogeneous $\zeta_\triangle$-invariant subsets that are not $\epsilon$ dense and decompose into a disjoint union of translates of $T^\kappa$. Again, we may assume $\kappa\in K$ following Remark \ref{switch12}.

Suppose the opposite is true, namely, there are infinitely many $d$-dimensional homogeneous $\zeta_\triangle$-invariant subsets consisting of cosets of $T^\kappa$ that are not $\epsilon$-dense. By the same argument as in Step 1, we may find a sequence of distinct $d$-dimensional homogeneous $\zeta_\triangle$-invariant subsets $L_h=\bigsqcup_{t=1}^{s_h}L_{h,t}$ that avoid a fixed ball $B_{\frac\epsilon2}(\bfy)$ of radius $\frac\epsilon2$, where each $L_{h,t}$ is a translate of $T^\kappa$. We remark that the $L_{h,t}$'s are all different, since if $L_{h,t}=L_{h',t'}$ for $h\neq h'$ then by topological transitivity of $L_h$ and $L_{h'}$ this would imply $L_h=L_{h'}$ which contradicts the way the $L_h$'s are chosen.

By Corollary \ref{kappacapinfty}, for each $L_{h,t}$ one can find a point $\bfx_{h,t}\in L_{h,t}\cap T^\infty$. Since $L_{h,t}=\bfx_{h,t}+T^\kappa$, the $\bfx_{h,t}$'s are all different and thus $\bigcup_{h=1}^\infty(L_h\cap T^\infty)$ is an infinite set. However by invariance of $L_h$ and $T^\infty$, this set is a $\zeta_\triangle$-invariant subset of $T^\infty$. Hence it follows from Lemma \ref{TkappaBerend} that $\bigcup_{h=1}^\infty(L_h\cap T^\infty)$ is dense in $T^\infty$.

Consider the center $\bfy$ of the ball $B_{\frac\epsilon2}(\bfy)$, by construction $B_{\frac\epsilon2}(\bfy)$ is disjoint from $L_h$ for all $h$. But by Corollary \ref{kappacapinfty} we can fix a point $\bfx \in(\bfy+T^\kappa)\cap T^\infty$. By density of $\bigcup_{h=1}^\infty(L_h\cap T^\infty)$, there are $h\in\bN$ and $\bfx'\in L_h\cap T^\infty$ such that $\|\bfx'-\bfx\|\leq\frac\epsilon2$. Then $\bfx'+(\bfy-\bfx)\in B_{\frac\epsilon2}(\bfy)$ and thus $\bfx'+(\bfy-\bfx)\notin L_h$. But on the other hand $\bfy-\bfx\in T^\kappa$ and $\bfx'+(\bfy-
\bfx)\in\bfx'+T^\kappa\subset L_h$, a contradiction. The proof is completed.
\end{proof}

\subsection{Restatement of the main results}
Finally we can restate the main result of this paper in the $X^2$ setting:

\begin{theorem}\label{CartanjoiningX2} Suppose $K$, $\Gamma$, $X$ and $\zeta$ satisfy the conditions in Proposition \ref{Gfield} and $\zeta_\triangle$ is constructed by (\ref{X2multi}), then:
\begin{enumerate}[leftmargin=*,label=\textup{(\arabic*)}]
\item If $r \geq 3$ and an infinite proper closed subset $A$ of $X^2$ is invariant and topologically transitive under the $\bZ^r$-action $\zeta_\triangle$, then $A$ is a $d$-dimensional homogeneous $\zeta_\triangle$-invariant subset. Any closed $\zeta _ \triangle$-invariant subset of $X^2$ is a finite union of topologically transitive closed  invariant subsets of $X^2$.

\item If $r=2$, then there exist a point $\bfx \in X^2$ and three different $d$-dimensional homogeneous $\zeta_\triangle$-invariant subtori $T^{\kappa_1},T^{\kappa_2},T^{\kappa_3}\subset(\bT^d)^2$  such that the $\zeta_\triangle$-orbit closure of $\bfx$ is a disjoint union:
\begin{equation}\label{CartanjoiningX2eq}\overline{\{ \zeta_\triangle^\bfn.\bfx:\bfn \in\bZ^r \}}=\{ \zeta_\triangle^\bfn.\bfx:\bfn \in\bZ^r \} \sqcup \big(\bigcup_{i=1}^3T^{\kappa_i}\big).\end{equation}
\end{enumerate}
\end{theorem}

\begin{theorem}\label{nondenserationalX2}In the same setting as above, suppose $r\geq 3$ and $\epsilon>\epsilon'>0$. Let $L_1,\cdots,L_q$ be all the $d$-dimensional homogeneous $\zeta_\triangle$-invariant subsets failing to be $\epsilon'$-dense in $X^2$. Then among all the torsion points whose $\zeta_\triangle$-orbits are not $\epsilon$-dense, there are only finitely many which are not contained in $\bigcup_{h=1}^qL_h$.

In particular, for all $\epsilon>0$, there is a finite union of $d$-dimensional homogeneous $\zeta_\triangle$-invariant subsets that contains all the torsion points whose $\zeta_\triangle$ -orbits are not $\epsilon$-dense in $X^2$.\end{theorem}

Here Lemma \ref{finitedhomo} is implicitly used to assert that there are only finitely many $L_h$'s which are not $\epsilon'$-dense.

\begin{lemma}\label{CartanjoiningT2X2}Theorem \ref{CartanjoiningX2} implies Theorem \ref{Cartanjoining}; and Theorem \ref{nondenserationalX2} implies Theorem \ref{nondenserational}.\end{lemma}

\begin{proof}Because the $\bZ^r$-actions $\zeta_\triangle$ and $\alpha_\triangle$ are conjugate via the continuous group isomorphism $\psi_\triangle$, to deduce Theorem \ref{Cartanjoining} from Theorem \ref{CartanjoiningX2} it suffices to observe that the preimage under $\psi$ of a $d$-dimensional homogeneous $\zeta_\triangle$-invariant subtorus is a $d$-dimensional $\alpha_\triangle$-invariant subtorus in $\bT^d$ by Lemma \ref{dhomotorprop}, and thus the preimage of a $d$-dimensional homogeneous $\zeta_\triangle$-invariant subset is a disjoint union of finitely many parallel translates of a $d$-dimensional $\alpha_\triangle$-invariant subtorus by Proposition \ref{dhomoprop}.

In addition to these ingredients, to get the implication from Theorem \ref{nondenserationalX2} to Theorem \ref{nondenserational} it is enough to note that under $\psi_\triangle$, any $\epsilon$-dense subset in $(\bT^d)^2$ becomes $C\epsilon$-dense in $X^2$ , where $C$ is a positive constant depending only on $\psi_\triangle$.\end{proof}

The proofs of Theorems \ref{CartanjoiningX2} and \ref{nondenserationalX2} are going to occupy the rest of paper.

\section{Any infinite invariant set has a homogeneous subset}\label{containhomo}

In order to get the first half of Theorem \ref{CartanjoiningX2}, our strategy is to prove first that the set $A$ contains a homogeneous invariant subset and then show that there are no other points in it. The aim of this section is to show the following:

\begin{proposition}\label{firstdhomo}Suppose $r\geq 2$ and $A\subset X^2$ is an infinite $\zeta_\triangle$-invariant closed subset, then there is a $d$-dimensional homogeneous $\zeta_\triangle$-invariant subset $L$ of $X^2$ such that $L\subset A$.\end{proposition}

Note that this proposition applies in both the $r\geq 3$ and $r=2$ cases.
The proof of Proposition \ref{firstdhomo} borrows ideas from \cite{B83}*{\S 4}

\subsection{Translate of a torsion point along an eigenspace}

In this part, we are going to show $A$ contains a torsion point $\bfz$ as well as another point in the $V_i$-foliation through $\bfz$ for some $i\in I$.

\begin{lemma}\label{containtorsion}If $A\subset X^2$ is an infinite $\zeta_\triangle$-closed subset then it has an accumulation point which is a torsion point in $X^2$.\end{lemma}
\begin{proof}$A$ is compact as a closed subset of the compact set $X^2$. Hence as it is infinite there must be an accumulation point $\bfy=(y^{(1)},y^{(2)})\in A$. If $\bfy$ is a torsion point then we are done, so we assume $\bfy$ is not of torsion.

Let $A'=\overline{\{\zeta_\triangle^\bfn.\bfy:\bfn\in\bZ^r\}}$. Then $A'\subset A$ and is also $\zeta_\triangle$-invariant. By Theorem \ref{Berend'} either the set $\{\zeta^\bfn.y^{(1)}:\bfn\in\bZ^r\}$ is dense in $X$ or $y^{(1)}$ is a torsion point. In both cases, there is a sequence of $\bfn_k$ such that $\zeta^{\bfn_k}.y^{(1)}$ converges to a torsion point $z^{(1)}$ (when $y^{(1)}$ is a torsion point, simply take $\bfn_k=\bfzero$ for all $k$). As $X$ is compact, by passing to a subsequence if necessary we may assume $\zeta^{\bfn_k}.y^{(2)}\rightarrow z_*^{(2)}$ for some $z_*^{(2)}\in X$. Then $\bfz_*=(z^{(1)},z_*^{(2)})=\lim_{k\rightarrow\infty}\zeta_\triangle^{\bfn_k}.\bfy\in A'$. Let $H$ be the stabilizer $\{\bfn\in \bZ^r:\zeta^\bfn.z^{(1)}=z^{(1)}\}$, which has finite index in $\bZ^r$ as $z^{(1)}$ is of torsion. It follows from Theorem \ref{Berend'} that $\overline{\{\zeta^\bfn.z_*^{(2)}:\bfn\in H\}}\subset X$ contains a torsion point, which we denote by $z^{(2)}$. Then there is a new sequence of $\bfm_k\in H$ such that $\zeta^{\bfm_k}.z_*^{(2)}$ converges to $z^{(2)}$ as $k\rightarrow\infty$. Then as $A'$ is a $\zeta_\triangle$-invariant closed set, it contains $\bfz:=(z^{(1)},z^{(2)})=\lim_{k\rightarrow\infty}\zeta_\triangle^{\bfm_k}.\bfz_*$, which is a torsion point in $X^2$.

Observe as $\bfz\in A'$ it is a limit point of the $\zeta_\triangle$-orbit of $\bfy$. But since $\bfz$ is of torsion and $\bfy$ is not, none of the points from the orbit of $\bfy$ coincide with $\bfz$. Thus $\bfz$ is an accumulation point of $A$. This completes the proof.\end{proof}

In particular, the lemma gives the following corollary.

\begin{corollary}When $r\geq 2$, any minimal $\zeta_\triangle$-invariant closed subset of $X^2$ is a finite set of torsion points.\end{corollary}

\begin{proof}Let $M$ be a minimal invariant set. By lemma, $M$ contains a torsion point $\bfz$. By minimality $M$ is the $\zeta_\triangle$-orbit of $\bfz$, which consists of finitely many torsion points.\end{proof}

Lemma \ref{containtorsion} tells us that there is a sequence of points in $A$ converging to a torsion point. We wish to have some control on the direction along which such convergence takes place.

\begin{lemma}\label{expansion}For any finite-index subgroup $H\leq \bZ^r$, there is a positive number $C=C(H)$ such that for all $i\in I$, there exists $\bfn\in\bZ^r$ such that $0<\lambda_i(\bfn)\leq C$ and $\lambda_j(\bfn)<0$ for all $j\in I\backslash\{i\}$. \end{lemma}
\begin{proof}It follows from Remark \ref{dirichlet} that any cone of non-empty interior in $W$ contains points from $\cL(H)$. In particular for any given index $i\in I$, the cone $\{(\lambda_j)_{j\in I}\in W:\lambda_i>0\text{ and }\lambda_j<0,\forall j\neq i\}$, which has open interior in $W$, contains a point $\cL(\bfn_i)$ where $\bfn_i\in H$. It suffices to take $C=\max_{i\in I}\lambda_i(\bfn_i)$.\end{proof}

\begin{corollary}\label{torsioneigen}If $A$ is an infinite $\zeta_\triangle$-invariant closed subset then there is an index $i\in I$ such that there are a torsion point $\bfz$ in $X^2$ and a non-zero vector $\bfv\in V_i^\square$ such that $A$ contains both $\bfz$ and $\bfz+\bfv$.\end{corollary}

\begin{proof}Let $\bfz\in A$ be given by Lemma \ref{containtorsion}. Then there is a sequence of non-zero vectors $\bfv_k\in(\bR^{r_1}\oplus\bC^{r_2})^2$ converging to $0$ as $k\rightarrow\infty$, such that $\bfz+\bfv_k\in A$. There is a unique decomposition $\bfv_k=\sum_{i\in I}(\bfv_k)_i$ where $(\bfv_k)_i\in V_i^\square$. As $I$ is finite, by reducing to a subsequence, we can find $i\in I$ such that $|(\bfv_k)_i|\geq|(\bfv_k)_j|$ for all $j\neq i$ and $k\in\bN$. In particular, $|(\bfv_k)_i|>0$.

Now let $H=\Stab_{\zeta_\triangle}(\bfz)$. As $H$ has finite index in $\bZ^r$, Lemma \ref{expansion} applies and we obtain $\bfn\in H$ such that $1<|\zeta_i^\bfn|=e^{\lambda_i(\bfn)}\leq e^C$ and $|\zeta_j^\bfn|<1$ for all $j\in I\backslash\{i\}$ where $C=C(H)$. For all sufficiently large $k$, there is a positive integer $l_k$ such that $|\zeta_i^\bfn|^{l_k}\cdot|(\bfv_k)_i|\in [1,e^C]$. Hence since $\zeta^\bfn$ acts as the multiplication by $\zeta_i^\bfn$ on $V_i^\square$, which is isomorphic to $\bR^2$ or $\bC^2$, \begin{equation}\zeta^{l_k\bfn}.(\bfv_k)_i=(\zeta_i^\bfn)^{l_k}(\bfv_k)_i\in\{\bfv\in V_i^\square:|\bfv|\in[1,e^C]\}.\end{equation} Because this is a compact subset in $V_i^\square$, by passing to a subsequence of $k\in\bN$, it is alright to assume the sequence $\{\zeta^{l_k\bfn}.(\bfv_k)_i\}$ converges to a limit vector $\bfv$. Then $\bfv\in V_i^\square$ and $|\bfv|\in[1,e^C]$.

Furthermore, \begin{equation}\label{torsioneigen1}\begin{split}|\zeta^{l_k\bfn}.\bfv_k-\zeta^{l_k\bfn}.(\bfv_k)_i|
=&\big|\sum_{j\in I\backslash\{i\}}\zeta^{l_k\bfn}.(\bfv_k)_j\big|
\leq\sum_{j\in I\backslash\{i\}}|(\zeta_j^\bfn)^{l_k}(\bfv_k)_j|\\\leq&\sum_{j\in I\backslash\{i\}}|(\bfv_k)_j|.\end{split}\end{equation}  Since $\bfv_k\rightarrow 0$, $\lim_{k\rightarrow\infty}|(\bfv_k)_j|=0$ for $j\neq i$. Hence (\ref{torsioneigen1}) decays to $0$ as $k\rightarrow\infty$ and $\lim_{k\rightarrow\infty}\zeta^{l_k\bfn}.\bfv_k=\lim_{k\rightarrow\infty}\zeta^{l_k\bfn}.(\bfv_k)_i=\bfv$.

Notice since $l_k\bfn\in H$, $\zeta_\triangle^{l_k\bfn}.\bfz=\bfz$. So $\zeta_\triangle^{l_k\bfn}.(\bfz+\bfv_k)$ is just $\bfz+\zeta^{l_k\bfn}.\bfv_k$ and converges to $\bfz+\bfv$. Since $\bfz+\bfv_k$ is in the $\zeta_\triangle$-invariant closed subset $A$, this implies $\bfz+\bfv\in A$ and completes the proof.\end{proof}

\subsection{Construction of an infinitely long line} We showed $A$ contains both $\bfz$ and $\bfz+\bfv$. Using invariance under the group action, a large subset of $V_i^\square$ can be constructed from the difference vector $\bfv\in V_i^\square$.

$\bfv\in V_i^{\square}$ can be accordingly written as $(v^{(1)},v^{(2)})$ where $v^{(1)},v^{(2)}$ are from $V_i$, which is isomorphic to $\bF_i=\bR$ or $\bC$ as a vector space. As at least one of $v^{(1)}$, $v^{(2)}$ is not zero, for the moment we are going to assume $v^{(1)}\neq 0$. Then there is a number $\kappa_*\in\bF_i$ such that $v^{(2)}=\kappa_*v^{(1)}$.

\begin{lemma}\label{line} Suppose $\bfz$ is a torsion point and $\bfv=(v,\kappa_*v)$ is a non-zero vector from $V_i^\square$ for some $i\in I$ where $v\in V_i$ and $\kappa_*\in\bF_i$. If a $\zeta_\triangle$-invariant closed subset $A$ contains both $\bfz$ and $\bfz+\bfv$, then there are $\bfy\in A$ and a non-zero vector $w\in V_i$ such that $\bfy+\rho\bfw\in A$ for all $\rho\in\bR$ where $\bfw=(w,\kappa_*w)\in V_i^\square$.\end{lemma}

We need the following fact to prove Lemma \ref{line}.

\begin{lemma}\label{smallLog}Suppose $r\geq 2$. Then for all finite-index subgroup $H\leq\bZ^r$, $i\in I$ and $\delta>0$, there exists $\bfn\in H\backslash\{0\}$ such that $|\Log\zeta_i^\bfn|<\delta$ where $\Log\zeta_i^\bfn\in\bC$ denotes the principal value of logarithm of $\zeta_i^\bfn$.\end{lemma}

\begin{proof} As $\rank(H)=r$, we can fix two linearly independent non-trivial elements $\bfm_1,\bfm_2\in H$. Set $\tau(p_1,p_2)=\big(\lambda_i(p_1\bfm_1+q_2\bfm_2),\arg\zeta_i^{p_1\bfm_1+q_2\bfm_2}\big)$, then $\tau$ is a group morphism from $\bZ^2$ to $\bR\oplus(\bR/2\pi\bZ)$. For all $p_1,p_2\in\{-P,P-1,\cdots,P-1,P\}$, $\tau(p_1,p_2)\in[-MP,MP]\oplus(\bR/2\pi\bZ)$ where $M=\max(|\lambda_i(\bfm_1)|, |\lambda_i(\bfm_2))|$. Because $[-MP,MP]\oplus(\bR/2\pi\bZ)$ can be covered by $O_{M,\delta}(P)$ balls of radius $\delta$ but we are considering $(2P+1)^2$ pairs of $(p_1,p_2)$, if $P$ is sufficiently large with respect to $M$ and $\delta$ then there are two distinct pairs $(p_1,p_2),(p'_1,p'_2)\in\{-P,P-1,\cdots,P-1,P\}^2$ such that $\tau(p_1,p_2)$ and $\tau(p'_1,p'_2)$ are in the same ball. Let $\bfn=(p_1-p'_1)\bfm_1+ (p_2-p'_2)\bfm_2$, then as $\bfm_1$ and $\bfm_2$ are linearly independent elements in $H$, $\bfn$ is a non-zero element of $H$. Moreover $(\lambda_i(\bfn),\arg\zeta_i^\bfn)=\tau(p_1,p_2)-\tau(p'_1,p'_2)$ is of distance $\delta$ from the origin in $\bR\oplus(\bR/2\pi\bZ)$. In other words, if $\Arg \zeta_i^\bfn$ denotes the principal value of complex argument of $\zeta_i^\bfn$, then $|\lambda_i(\bfn)+\imag\Arg \zeta_i^\bfn|\leq\delta$. However $\lambda_i(\bfn)+\imag\Arg \zeta_i^\bfn$ is just $\Log\zeta_i^\bfn$ as $\lambda_i(\bfn)=\log|\zeta_i^\bfn)|$.\end{proof}

We show first a finitary statement, then take limit to get Lemma \ref{line}.

\begin{lemma}\label{finiteline} Let $A$, $i$, $\bfz$, $\bfv$ and $\kappa_*$ be as in Lemma \ref{line}. Then $\forall R>0$, $\forall\epsilon>0$, $\exists\bfy\in A$, $\exists w\in V_i$ such that $|w|=1$ and $A\cap B_\epsilon(\bfy+\rho\bfw)\neq\emptyset,\forall\rho\in[-R,R]$ where $\bfw=(w,\kappa_*w)\in V_i^\square$.\end{lemma}
\begin{proof} Without loss of generality, assume $R>\epsilon$.

Again let $H$ be the finite-index subgroup $\Stab_{\zeta_\triangle}(\bfz)\leq \bZ^r$. By Lemma \ref{expansion}, there exists $\bfm\in H$ such that $\lambda_i(\bfm)\in(0,C]$ where $C=C(H)$. Then as $v\in V_i$, for all integer $l$, $\zeta^{l\bfm}.v=\zeta_i^{l\bfm}.v$ and has length $e^{l\lambda_i(\bfm)}|v|$. We may choose $l\in\bZ$ such that \begin{equation}\label{finiteline1}\frac {2e^C\sqrt{1+|\kappa_*|^2}\cdot R^2}\epsilon<|\zeta_i^{l\bfm}v|\leq \frac {2e^{2C}\sqrt{1+|\kappa_*|^2}\cdot R^2}\epsilon.\end{equation}

Applying Lemma \ref{smallLog}, pick a nontrivial element  $\bfn\in H$ such that \begin{equation}\label{finiteline2}0 \neq |\Log\zeta_i^\bfn|<\frac{\epsilon^2}{4e^{2C}\sqrt{1+|\kappa_*|^2}\cdot R^2}.\end{equation}

Define \[\bfy=\zeta_\triangle^{l\bfm}.(\bfz+\bfv)=\bfz+\zeta_i^{l\bfm}\bfv, \quad w=\frac{(\Log\zeta_i^\bfn)\zeta_i^{l\bfm}v}{|(\Log\zeta_i^\bfn)\zeta_i^{l\bfm}v|}\] and denote $(w,\kappa_*w)$ by $\bfw$. 


For all $\rho\in [-R,R]$, take the integer \begin{equation}\label{finiteline3}t=\sign(\rho)\cdot\left\lfloor\frac{|\rho|}{|(\Log\zeta_i^\bfn)\zeta_i^{l\bfm}v|}\right\rfloor.\end{equation}  We claim \begin{equation}\label{finitelineclaim}\left\|\zeta_\triangle^{l\bfm+t\bfn}.(\bfz+\bfv)-(\bfy+\rho\bfw)\right\|\leq\epsilon.\end{equation}
To see this, first observe that as $\bfm,\bfn\in H$,
\begin{equation}\begin{split}\zeta_\triangle^{l\bfm+t\bfn}.(\bfz+\bfv)=&
\bfz+\zeta_i^{l\bfm+t\bfn}\bfv\\
=&\bfz+e^{t\Log\zeta_i^\bfn}\zeta_i^{l\bfm}\bfv\\
=&\bfz+\zeta_i^{l\bfm}\bfv+t(\Log\zeta_i^\bfn)\zeta_i^{l\bfm}\bfv\\
&\ \ +\sum_{k=2}^\infty\frac{(t\Log\zeta_i^\bfn)^k}{k!}\cdot\zeta_i^{l\bfm}\bfv\\
=&\bfy+\rho\bfw+\left(t-\frac{\rho}{|(\Log\zeta_i^\bfn)\zeta_i^{l\bfm}v|}\right)(\Log\zeta_i^\bfn)\zeta_i^{l\bfm}\bfv\\
&\ \ +\sum_{k=2}^\infty\frac{(t\Log\zeta_i^\bfn)^k}{k!}\cdot\zeta_i^{l\bfm}\bfv,\end{split}.\end{equation}
Therefore it suffices to show
\begin{equation}\label{finiteline5}\left|\left(t-\frac{\rho}{|(\Log\zeta_i^\bfn)\zeta_i^{l\bfm}v|}\right)\cdot(\Log\zeta_i^\bfn)\zeta_i^{l\bfm}\bfv\right|\leq\frac\epsilon 2\end{equation}
and \begin{equation}\label{finiteline6}\left|\sum_{k=2}^\infty\frac{(t\Log\zeta_i^\bfn)^k}{k!}\cdot\zeta_i^{l\bfm}\bfv\right|\leq\frac\epsilon 2.\end{equation}

By choice of $t$, $\big|t-\frac{\rho}{|(\Log\zeta_i^\bfn)\zeta_i^{l\bfm}v|}\big|\leq 1$. Moreover by (\ref{finiteline1}) and (\ref{finiteline2}),  \[|(\Log\zeta_i^\bfn)\zeta_i^{l\bfm}\bfv|\leq\frac{\epsilon^2}{4e^{2C}\sqrt{1+|\kappa_*|^2}\cdot R^2}\cdot\frac{2e^{2C}\sqrt{1+|\kappa_*|^2}\cdot R^2}\epsilon=\frac\epsilon{2},\] 
establishing (\ref{finiteline5}).

On the other hand, 
 \begin{equation}\begin{split}|t\Log\zeta_i^\bfn|\leq& \frac{|\rho|}{|(\Log\zeta_i^\bfn)\zeta_i^{l\bfm}v|}\cdot|\Log\zeta_i^\bfn|=\frac{|\rho|}{|\zeta_i^{l\bfm}v|}\\
\leq&R\left(\frac{2e^C\sqrt{1+|\kappa_*|^2}\cdot R^2}\epsilon\right)^{-1}\\
=&\frac{\epsilon}{2e^C\sqrt{1+|\kappa_*|^2}\cdot R}\end{split}\end{equation}
As we assumed $R>\epsilon$, this expression is bounded by $1$. Thus \begin{equation}\label{finiteline7}\begin{split}\left|\sum_{k=2}^\infty\frac{(t\Log\zeta_i^\bfn)^k}{k!}\right|=&\left|\sum_{k=2}^\infty\frac{(t\Log\zeta_i^\bfn)^{k-2}}{k!}\right|\cdot|t\Log\zeta_i^\bfn|^2\\
\leq&\left(\sum_{k=2}^\infty\frac1{k!}\right)\cdot|t\Log\zeta_i^\bfn|^2\\
\leq&|t\Log\zeta_i^\bfn|^2\leq\left(\frac{\epsilon}{2e^C\sqrt{1+|\kappa_*|^2}\cdot R}\right)^2\\
\leq&\frac{\epsilon^2}{4e^{2C}\sqrt{1+|\kappa_*|^2}R^2}.\end{split}\end{equation}
(\ref{finiteline6}) is obtained by taking the product of (\ref{finiteline1}) and (\ref{finiteline7}). This shows the claim (\ref{finitelineclaim}). The lemma follows as $\zeta_\triangle^{l\bfm+t\bfn}.(\bfz+\bfv)\in A$.\end{proof}

\begin{proof}[Proof of Lemma \ref{line}] By Lemma \ref{finiteline}, for all $n\in\bN$, there are $\bfy_n\in X^2$ and $w_n\in V_i$ with $|w_n|=1$ such that if we denote $\bfw_n=(w_n,\kappa_*w_n)\in V_i^\square$ then there exists a point $\bfy_{n,\rho}$ in $A\cap B_{\frac1n}(\bfy_n+\rho\bfw_n)$ for all $\rho\in[-n,n]$.

Since both $X^2$ and the unit circle in $V_i$ (which is just $\{\pm1\}$ if $V_i\cong\bR$) are compact, there is a subsequence $\{n_k\}_{k=1}^\infty$ such that as $k$ tends to $\infty$, $\bfy_{n_k}\rightarrow\bfy$ and $w_{n_k}\rightarrow w$  for some limits $\bfy\in X^2$ and $w\in V_i$ with $|w|=1$.

Set $\bfw=(w,\kappa_*w)\in V_i^\square$. Now for any $\rho\in\bR$ and $\epsilon>0$, choose $k$ such that $n_k\geq\max(\rho,\frac3\epsilon)$, $\|\bfy_{n_k}-\bfy\|\leq\frac\epsilon3$ and $|w_{n_k}-w|\leq\frac\epsilon{3|\rho|\sqrt{1+|\kappa_*|^2}}$. Then $|\bfw_{n_k}-\bfw|=|(w_{n_k}-w,\kappa_*(w_{n_k}-w))|=\sqrt{1+|\kappa_*|^2}\cdot|w_{n_k}-w|\leq\frac{\epsilon}{3|
\rho|}$. By assumption, there exists a point $\bfy_{n_k,\rho}\in A$ within distance $\frac1{n_k}$ from $\bfy_{n_k}+\rho\bfw_{n_k}$. Then the distance between $\bfy_{n_k,\rho}$ and $\bfy+\rho\bfw$ is bounded by \begin{equation}\begin{split}&\|\bfy_{n_k,\rho}-(\bfy_{n_k}+\rho\bfw_{n_k})\|+\|(\bfy_{n_k}+\rho\bfw_{n_k})-(\bfy+\rho\bfw)\|\\
\leq&\frac1{n_k}+\big(\|\bfy_{n_k}-\bfy\|+|\rho|\cdot|\bfw_{n_k}-\bfw|\big)\\
\leq&(\frac3\epsilon)^{-1}+\frac\epsilon3+|\rho|\cdot\frac\epsilon{3|\rho|}=\epsilon.\end{split}\end{equation} The proof is completed. \end{proof}

\subsection{Construction of a homogeneous invariant set}
\begin{lemma}\label{denseorcontaindhomo}Suppose a closed subset $A\subset X^2$ is $\zeta_\triangle$-invariant and contains two points $\bfz$ and $\bfz+\bfv$ where $\bfz$ is a torsion point and $\bfv$ is a non-zero vector from $V_i^\square$ for some $i\in I$. Then:
\begin{enumerate}[label=\textup{(\arabic*)},leftmargin=*]
\item  If $\bfv\in V_i^\kappa$ for some $\kappa\in K\cup\{\infty\}$ then the $d$-dimensional homogeneous $\zeta_\triangle$-invariant subset $L=\{\zeta_\triangle^\bfn.\bfz':\bfn\in\bZ^r,\bfz'\in\bfz+T^\kappa\}$, which contains both $\bfz$ and $\bfz+\bfv$, is a subset of $A$.

\item If there does not exist such a slope $\kappa$ then $A=X^2$.
\end{enumerate}
\end{lemma}

For the construction of $V_i^\kappa$, see Definition \ref{Vkappa}.

\begin{proof}\noindent(1) Suppose $\bfv\in V_i^\kappa$ where $\kappa\in K\cup\{\infty\}$. Then the set $L$ is a $d$-dimensional homogeneous $\zeta_\triangle$-invariant subset and obviously $\bfz\in L$. By Lemma \ref{Tkappatangent}, $\bfz+\bfv\in \bfz+T^\kappa\subset L$ as $\bfv\in V^\kappa$.

We want to show from the fact $\bfz+\bfv\in L$ that $L\subset A$. As $A$ is closed and $\zeta_\triangle$-invariant, using Proposition \ref{dhomoprop}.(3) it suffices to prove $\bfz+\bfv$ is not a torsion point. As $\bfz$ is of torsion, this reduces to prove that $\pi_\triangle(\bfv)=\big(\pi(v^{(1)}),\pi(v^{(2)})\big)$ is not a torsion point where $\bfv=(v^{(1)},v^{(2)})$. Because $\bfv\neq 0$, there is at least one of $v^{(1)}$, $v^{(2)}$ that doesn't vanish, so it is enough to show that $\pi(v)\in X$ is not a torsion point of $X$ for any non-zero vector $v\in V_i$. Suppose this is not true then there is a non-zero integer $q$ such that $qv\in\Gamma$. But $\Gamma\subset\sigma(K)$, hence $qv=\sigma(\gamma)$ where $\gamma\in K$. For all $j\in I\backslash\{i\}$, as the $V_j$ coordinate of $qv\in V_i$ vanishes, $\sigma_j(\gamma)=0$ (because $r_1+r_2=r+1\geq 3$, there is always another index $j\neq i$ in $I$). But this can happen only if $\gamma=0$, in consequence $qv=0$ and thus $v=0$, which provides a contradiction. This shows part (1) of the lemma.\\

\noindent(2) Suppose $\bfv\notin V_i^\kappa$ for all $\kappa\in K\cup\{\infty\}$. As before, let $\bfv=(v^{(1)},v^{(2)})$ where $v^{(1)},v^{(2)}\in V_i$, which is isomorphic to either $\bR$ or $\bC$. As $\bfv\notin V_i^\infty=\{0\}\times V_i$, $v^{(1)}\neq 0$. So there is a number $\kappa_*$ from either $\bR$ or $\bC$, whichever $V_i$ is isomorphic to, such that $v^{(2)}=\kappa_*v^{(1)}$.

By the assumption on $\bfv$, $\kappa_*\notin\sigma_i(K)$. In fact by Definition \ref{Vkappa}, if $\kappa_*=\sigma_i(\kappa)$ for some $\kappa\in K$ then $\bfv\in V_i^\kappa$.

By Lemma \ref{line}, $A$ contains an infinitely long line $\{\bfy+\rho\bfw:\rho\in\bR\}$ where $\bfy\in X^2$ and $\bfw$ can be written as $(w,\kappa_*w)$ for some non-zero vector $w\in V_i$. Furthermore by Remark \ref{totalirrrmk}, there is an $\bfn$ such that for all $l\in\bZ\backslash\{0\}$, $\zeta^{l\bfn}$ is not in any proper subfield of $K$. It follows from Lemma \ref{rotatedense}, which is stated and proved below, that the set $\bigcup_{l=1}^\infty\{\zeta_\triangle^{l\bfn}.(\bfy+\rho\bfw):\rho\in\bR\}$ is dense in $X^2$. Because this union is contained in the $\zeta_\triangle$-invariant subset $A$,  $A$ must be $X^2$.
\end{proof}

\begin{lemma}\label{rotatedense}Suppose either $1\leq i\leq r_1$ and $\kappa_*\in\bR$, or $r_1<i\leq r_1+r_2$ and $\kappa_*\in\bC$. Furthermore, assume $\kappa_*\notin\sigma_i(K)$. For $\bfy\in X^2$ and  $w\in V_i\backslash\{0\}$, denote $\bfw=(w,\kappa_*w)$. Suppose in addition $\bfn\in\bZ^r$ satisfies that $\zeta^{l\bfn}$ is not in any proper subfield of $K$ for all non-zero integer $l$, then $\bigcup_{l=1}^\infty\{\zeta_\triangle^{l\bfn}.(\bfy+\rho\bfw):\rho\in\bR\}$ is dense in $X^2$.\end{lemma}

\begin{proof} Assume the union is not dense, then there is a ball $\mathring B_\epsilon(\bfx)\subset X^2$ which is disjoint from $\{\zeta_\triangle^{l\bfn}.(\bfy+\rho\bfw):\rho\in\bR\}$ for all $l\in\bN$. Since $X^2$ is compact, it can be covered by finitely many balls of radius $\frac\epsilon2$. In consequence there is a ball $\mathring B_{\frac\epsilon2}(\bfx')$ that contains $\zeta_\triangle^{l_k\bfn}.\bfy$ for a subsequence $l_k$. Since $\zeta_\triangle^{l_k\bfn}.(\bfy+\rho\bfw)=\zeta_\triangle^{l_k\bfn}.\bfy+\zeta_\triangle^{l_k\bfn}.\rho\bfw$, it follows that for each $l_k$, $\{\zeta_\triangle^{l_k\bfn}.\rho\bfw:\rho\in\bR\}$ is disjoint from $\mathring B_{\frac\epsilon2}(\bfx-\bfx')$. Hence the closure $P_{l_k}:=\overline{\{\zeta_\triangle^{l_k\bfn}.\rho\bfw:\rho\in\bR\}}$ avoids $\mathring B_{\frac\epsilon2}(\bfx-\bfx')$, in particular the sequence of subsets $\{P_{l_k}\}_{k=1}^\infty$ does not converge to $X^2$ in the Hausdorff metric.

On the other hand, the sets $P_{l_k}$ are closed subgroups of $X^2$. Hence by \cite{B83}*{Lemma 4.7}, there is a non-trivial character $\bfxi\in\widehat{X^2}$ which is in the annihilator of $P_{l_k}$ for all sufficiently large $k$. In particular \begin{equation}\label{rotatedense0}\langle\bfxi,\pi_\triangle(\zeta_\triangle^{l_k\bfn}.\rho\bfw)\rangle=0\text{ (mod }\bZ\text{)},\forall\rho\in\bR\end{equation} for all $k$ large enough.

Since $w\in V_i$, when regarded as a vector in $\bR^{r_1}\oplus\bC^{r_2}$, its $i$-th coordinate is a non-zero number $w_i\in\bF_i=\bR$ or $\bC$, while all other coordinates $w_j$'s are equal to $0$.

Write $\bfxi=(\xi^{(1)},\xi^{(2)})$. We claim that $\xi_i^{(1)}+\kappa_*\xi_i^{(2)}=0$. To prove this we distinguish between real and complex $V_i$'s.\newline

\noindent{\bf Case 1.} When $i\leq r_1$, $V_i\cong\bR$. Fix an arbitrary $l_k$ which is large enough, then (\ref{rotatedense0}) holds. By the duality formulas (\ref{fielddualformula}) and (\ref{X2dualformula}), \begin{equation}\label{rotatedense1}\begin{split}\langle\bfxi,\pi_\triangle(\zeta_\triangle^{l_k\bfn}.\rho\bfw)\rangle=&\langle\xi^{(1)},\pi(\zeta^{l_k\bfn}.\rho w)\rangle+\langle\xi^{(2)},\pi(\zeta^{l_k\bfn}.\rho\kappa_*w)\rangle\\
=&(\xi_i^{(1)}\zeta_i^{l_k\bfn}\rho w_i+ \xi_i^{(2)}\zeta_i^{l_k\bfn}\rho\kappa_*w_i\text{ mod }\bZ)\\
=&\left(\rho\big(\xi_i^{(1)}+\kappa_*\xi_i^{(2)}\big)\zeta_i^{l_k\bfn}w_i\text{ mod }\bZ\right).\end{split}\end{equation} Since $w_i\neq 0$, $\zeta_i^{l_k\bfn}\neq 0$ (as it is an algebraic unit) and $\rho$ is an arbitrary real number, (\ref{rotatedense0}) cannot be true unless $\xi_i^{(1)}+\kappa_*\xi_i^{(2)}=0$.\newline

\noindent{\bf Case 2.} Suppose now $r_1<i\leq r_1+r_2$, in which case $\sigma_i$ is a complex embedding and $V_i\cong\bC$. Then by (\ref{fielddualformula}) and (\ref{X2dualformula}),
\begin{equation}\label{rotatedense2}\begin{split}\langle\bfxi,\pi_\triangle(\zeta_\triangle^{l_k\bfn}.\rho\bfw)\rangle=&\big(\re(\xi_i^{(1)}\zeta_i^{l_k\bfn}\rho w_i)+\re(\xi_i^{(2)} \zeta_i^{l_k\bfn}\rho\kappa_*w_i)\text{ mod }\bZ\big)\\
=&\Big(\rho\re\big((\xi_i^{(1)}+\kappa_*\xi_i^{(2)})\zeta_i^{l_k\bfn}w_i\big)\text{ mod }\bZ\Big).\end{split}\end{equation}

Hence for any sufficiently large $k$, it follows from (\ref{rotatedense0}) that the expression $\re\big((\xi_i^{(1)}+\kappa_*\xi_i^{(2)})\zeta_i^{l_k\bfn}w_i\big)$ must vanish.

Assume first $\xi_i^{(1)}+\kappa_*\xi_i^{(2)}\neq 0$, then there is a non-zero number $u\in\bC$ such that for $z\in\bC$, $\re\big((\xi_i^{(1)}+\kappa_*\xi_i^{(2)})z\big)=0$ if and only if $\frac zu\in\bR$. Hence $\zeta_i^{l_k\bfn}w_i\in\bR u$ for large $k$. In consequence, if we fix two different large terms $l_{k_1}$ and $l_{k_2}$ from the subsequence $\{l_k\}$, then  $\zeta_i^{(l_{k_1}-l_{k_2})\bfn}=\dfrac{\zeta_i^{l_{k_1}\bfn}w_i}{\zeta_i^{l_{k_2}\bfn}w_i}\in\bR$. In other words, $\zeta^{(l_{k_1}-l_{k_2})\bfn}\in\sigma_i^{-1}(\bR)$. But since $\sigma_i$ is a complex embedding, $\sigma_i^{-1}(\bR)$ is a proper subfield of $K$. Because $l_{k_1}\neq l_{k_2}$, this contradicts the assumption on $\bfn$. Therefore $\xi_i^{(1)}+\kappa_*\xi_i^{(2)}=0$.\newline

So we proved that $\xi_i^{(1)}+\kappa_*\xi_i^{(2)}=0$ always holds. Recall $\xi^{(1)},\xi^{(2)}\in\hat X$, and thus, by Lemma \ref{XdualK}, can be respectively represented by vectors $\sigma(\theta^{(1)}),\sigma(\theta^{(2)})\in\bR^{r_1}\oplus\bC^{r_2}$ where $\theta^{(1)},\theta^{(2)}\in K$.

Notice $\theta^{(2)}\neq 0$. In fact, otherwise $\xi_i^{(2)}=\sigma_i(\theta^{(2)})$ and $\xi_i^{(1)}=-\kappa_*\sigma(\theta^{(2)})$ are both zero; since $\xi_i^{(1)}=\sigma_i(\theta^{(1)})$ we know $\theta^{(1)}=0$ as well. Thus $\bfxi=\big(\sigma(\theta^{(1)}),\sigma(\theta^{(2)})\big)$ is the trivial character on $X^2$, which contradicts the choice of $\bfxi$.

So $\kappa_*=-\frac{\xi_i^{(1)}}{\xi_i^{(2)}}=\sigma_i(-\frac{\theta^{(1)}}{\theta^{(2)}})\in\sigma_i(K)$. This produces a contradiction to the hypothesis on $\kappa_*$, hence establishes the lemma.\end{proof}

Now we show the main proposition of this section.

\begin{proof}[Proof of Proposition \ref{firstdhomo}] Given an infinite $\zeta_\triangle$-invariant closed subset $A$, by Corollary \ref{torsioneigen} there is an index $i\in I$ satisfying the condition in Lemma \ref{denseorcontaindhomo}. By applying the lemma, we see $A$ always contains a $d$-dimensional homogeneous $\zeta_\triangle$-invariant subset.\end{proof}

\section{Actions of rank $3$ or higher}\label{rank3}

We now establish the first half of Theorem \ref{Cartanjoining} as well as Theorem \ref{nondenserational}. By Lemma \ref{CartanjoiningT2X2} it suffices to prove Theorem \ref{CartanjoiningX2}.(1) and Theorem \ref{nondenserationalX2}.

Throughout this section, assume $r\geq 3$ and $A\subset X^2$ satisfies the following assumption:

\begin{condition}\label{nonisocond}$A$ is an infinite $\zeta_\triangle$-invariant subset and for any $d$-dimensional homogeneous $\zeta_\triangle$-invariant subset $L$ contained in $A$, $A\backslash L$ is dense in $A$.\end{condition}

We are going to show that:

\begin{proposition}\label{nonisofull}
When $r\geq 3$, if $A$ satisfies Condition \ref{nonisocond}, then $A=X^2$.
\end{proposition}

To begin with, we claim an inductive fact:

\begin{lemma}\label{nonisoinduction}Suppose $r\geq 3$ and $A$ satisfies Condition \ref{nonisocond}. If $A$ contains a finite but non-empty collection of $d$-dimensional homogeneous $\zeta_\triangle$-invariant subsets $L_0,\cdots,L_q$, then it contains at least one more $d$-dimensional homogeneous $\zeta_\triangle$-invariant set $L_{q+1}$ which is different from any of these.\end{lemma}

The proof of the lemma is the topic of subsections \ref{subsecselfreturn}--\ref{extra}. For now we prove Proposition \ref{nonisofull} assuming the lemma.

\begin{proof}[Proof of Proposition \ref{nonisofull}]By Proposition \ref{firstdhomo}, $A$ contains at least one $d$-dimensional homogeneous $\zeta_\triangle$-invariant subset. Fix such a subset $L_0$. By applying Lemma \ref{nonisoinduction} repeatedly, we see $A$ contains infinitely many $d$-dimensional homogeneous $\zeta_\triangle$-invariant subsets. By Lemma \ref{finitedhomo}, $A$ is $\epsilon$-dense in $X^2$ for any $\epsilon>0$. Since $A$ is closed, it must be $X^2$.\end{proof}

\subsection{Notation and preliminary observations}
From now on, suppose $q\geq 0$ and $L_0,L_1,L_2,\cdots,L_q$ are $d$-dimensional homogeneous $\zeta_\triangle$-invariant subsets, all different from each other and contained in $A$.  Then there are $\kappa_0,\kappa_1,\cdots,\kappa_q\in K\cup\{\infty\}$ such that for all $h=0,..,q$, $L_h=\bigsqcup_{t=1}^{s_h}L_{h,t}$ where $s_h\geq 1$ and each $L_{h,t}$ is a translate of $T^{\kappa_h}$ by a torsion element of $X^2$.

\begin{lemma}\label{finiteinter}
Set $E=\big(\bigcup_{h=1}^qL_h\big)\cap L_{0,1}$. Then $E$ is a finite (possibly empty) set of torsion points. \end{lemma}
\begin{proof}For the finiteness of $E=\bigcup_{h=1}^q\bigcup_{t=1}^{s_h}(L_{h,t}\cap L_{0,1})$, it suffices to show $L_{h,t}\cap L_{0,1}$ is finite for all pairs $(h,t)$ with $h\geq 1$. Recall $L_{h,t}$ is a translate of $T^{\kappa_h}$. If $\kappa_h\neq\kappa_0$ then the finiteness is claimed by Corollary \ref{finiteTkappainter}. If $\kappa_h=\kappa_0$ then $L_{h,t}$ and $L_{0,1}$ are parallel, hence have a non-empty intersection if and only if they coincide with each other. But this would imply $L_h=L_0$ as $L_h$ and $L_0$ are respectively the full $\zeta_\triangle$-orbits of $L_{h,t}$ and $L_{0,1}$. By choice $L_h\neq L_0$, hence $L_{h,t}\cap L_{0,1}=\emptyset$ when $\kappa_h=\kappa_0$.

%
By construction, $E$ is $\zeta _ \Delta$-invariant, and have established that it is finite. The lemma follows from the observation  that every finite $\zeta _ \Delta$-invariant subset of $X ^ 2$ can contain only torsion points.
\end{proof}

Define \begin{equation}\label{stab01}H=\left\{\begin{aligned}&\bigcap_{\bfx\in E}\Stab_{\zeta_\triangle}(\bfx),&&\text{ if }E\neq\emptyset;\\&\Stab_{\zeta_\triangle}(\bfz_{0,1}),&&\text{ if } E=\emptyset.\end{aligned}\right.\end{equation} where $\bfz_{0,1}$ is any torsion point in $L_{0,1}$.
The group $H$ preserves $L_{0,1}$ under $\zeta_\triangle$ in both cases, since if $H$ stabilizes any $\bfx\in L_{0,1}$ it also stabilizes $L_{0,1}=\bfx+T^{\kappa_0}$. Moreover, $H$ is a finite index subgroup of $\bZ^r$: When $E$ is empty this is clear. Suppose $E\neq\emptyset$, then as each $\bfx$ is of torsion, all the $\Stab_{\zeta_\triangle}(\bfx)$ have finite index in $\bZ^r$, and hence so does their intersection.
Clearly $E$ is preserved by the subgroup $H$ under $\zeta_\triangle$.

Let $\rho_0>0$ be small enough such that if $0<\theta<\rho_0$ the set \begin{equation}E_\theta=\bigcup_{\bfx\in E}(\mathring B_\theta(\bfx)\cap L_{0,1})\end{equation} is equal to $\{\bfx+\bfv:\bfx\in E,\bfv\in V^{\kappa_0},|\bfv|<\theta\}$, the open $\theta$-neighborhood of $E$ in $L_{0,1}$.

\subsection{A recurrence property}\label{subsecselfreturn}

 In this section we show that when $\theta$ is sufficiently small, the complement $L_{0,1}\backslash E_\theta$ has the following recurrence property:

\begin{lemma}\label{selfreturn}There are positive constants $C$ and $\theta_0$, which may depend on $L_{0,1}$ and $E$,  such that for all $i\in I$, $\theta\in(0,\theta_0)$, and $\bfx\in L_{0,1}\backslash E_\theta$, there exists  $\bfn\in H$ such that:

(1) $C<\lambda_i(\bfn)<2C$ and $\lambda_j(\bfn)<\frac C2$ for all $j\in I\backslash\{i\}$;

(2) $\zeta_\triangle^\bfn.\bfx\in L_{0,1}\backslash E_\theta$.\end{lemma}

\begin{proof}The proof has two steps:\newline

\noindent{\bf Step 1.} We show that there exist $C>2\log(r_1+r_2)$ and $\bfn_{ik}\in H$ for all pairs $i,k\in I$ with $i\neq k$ such that \begin{equation}\label{selfreturn0}\left\{\begin{aligned}&C<\lambda_i(\bfn_{ik})<2C,\\ &\frac C4<\lambda_k(\bfn_{ik})<\frac C2,\\&-4C<\lambda_j(\bfn_{ik})<\frac C2,\forall j\in I\backslash\{i,k\}.\end{aligned}\right.\end{equation}

This claim follows from two facts. The first one is that by  Remark~\ref{dirichlet},\ $\cL(H)$ is a full-rank lattice in $W$ (with $\cL$ and $W$ as in \S\ref{Xnota}), and hence there exists $b>0$ such that any ball of radius $b$ in $W$ contains at least one point from~$\cL(H)$.

Second, consider the subset \begin{equation}\label{selfreturn1}\begin{split}\Omega_{ik}=&\{( w_l)_{l\in I}:\sum_{l\in I}d_l w_l=0,  w_l\in(1,2), \\
&\ \ \ \  w_k\in(\frac14,\frac12),  w_j\in(-4,\frac12)\text{ for any other }j\} \subset W.\end{split}\end{equation}
 Recall that $d_i$ is the real dimension of $V_i$, which is either $1$ or $2$ depending on whether $\bF_i=\bR$ or $\bC$. Then $\Omega_{ik}$ is open, as it is defined by open conditions, and non-empty, as it contains the point given by $ w_i=\frac43$, $ w_k=\frac13$ and $ w_j=-\frac{4d_i+d_k}{3(d-d_i-d_k)}$ for all the other $j$'s (note this is possible because $d_i,d_k\leq 2$ and $d-d_i-d_k=\sum_{j\in I\backslash\{i,k\}}d_j\geq |I|-2=(r+1)-2\geq 1$ as long as $r\geq 2$).

Pick a large constant $C$ such that $C>2\log(r_1+r_2)$ and $C\Omega_{ik}$ contains a ball of radius $b$ for all pairs $(i,k)$. Thus by the choice of $b$, there is $\bfn_{ik}$ such that $\cL(\bfn_{ik})\in C\Omega_{ik}$. By definition of $\cL$, this is equivalent to the inequalities (\ref{selfreturn0}).\newline

\noindent{\bf Step 2.} We now claim the following:

\begin{itemize}
\item [($\star$)]\emph{There is $\theta_0>0$ such that for all $\theta\in(0,\theta_0)$, $\bfx\in L_{0,1}\backslash E_\theta$, and $i\in I$, $\zeta_\triangle^{\bfn_{ik}}.\bfx\in L_{0,1}\backslash E_\theta$ for at least one $k\in I\backslash\{i\}$. }
\end{itemize}
\noindent
Clearly, together with (\ref{selfreturn0}) this would imply the lemma.
Let
 \begin{equation}\label{selfreturn2}\rho=\min\big(\min_{\substack{\bfgamma\in\Gamma^2\\\bfgamma\neq\bfzero}}|\bfgamma|,\min_{\substack{\bfy,\bfy'\in E\\\bfy\neq\bfy'}}\|\bfy-\bfy'\|,\rho_0\big)>0.\end{equation} It is easy to verify that if $\bfv\in V^{\kappa_0}$ satisfies $0<|\bfv|<\rho$ then $\forall\bfy\in E$,\ $\bfy+\bfv\notin E$.

We prove the claim ($\star$) for \begin{equation}\theta_0=\frac\rho{2e^{6C}}.\end{equation}

Suppose the claim fails for some $\bfx\in L_{0,1}\backslash E_\theta$ and $i\in I$ where $\theta<\theta_0$, i.e.\ that $\zeta_\triangle^{\bfn_{ik}}.\bfx\notin L_{0,1}\backslash E_\theta$ for all~$k\neq i$, which implies that $\zeta_\triangle^{\bfn_{ik}}.\bfx\in E_\theta$. In other words there exists $\bfy_k\in E$ such that $\|\zeta_\triangle^{\bfn_{ik}}.\bfx-\bfy_k\|<\theta$. In this case $E\neq\emptyset$, hence because $\bfn_{ik}\in H$, $\zeta_\triangle^{-\bfn_{ik}}.\bfy_k=\bfy_k$ by construction of~$H$. Thus \begin{equation}\label{x-yk}\begin{split}\|\bfx-\bfy_k\|=&\|\zeta_\triangle^{-\bfn_{ik}}.(\zeta_\triangle^{\bfn_{ik}}.\bfx-\bfy_k)\|\\
\leq&e^{\max_{j\in I}(-\lambda_j(\bfn_{ik}))}\|\zeta_\triangle^{\bfn_{ik}}.\bfx-\bfy_k\|\text{\ \  (by (\ref{X2Lip}))}\\
< &e^{4C}\theta<e^{4C}\theta_0\\
=&\frac\rho{2e^{2C}}\end{split}\end{equation}
So for distinct $k, k'\in I\backslash\{i\}$, \begin{equation}\|\bfy_k-\bfy_{k'}\|\leq\|\bfx-\bfy_k\|+\|\bfx-\bfy_{k'}\|<\rho\leq \min_{\substack{\bfy,\bfy'\in E\\\bfy\neq\bfy'}}\|\bfy-\bfy'\|.\end{equation} 
Therefore all the $\bfy_k \in E$ are the same; denote this common value by~$\bfy$. In particular, $\|\zeta_\triangle^{\bfn_{ik}}.\bfx-\bfy\|<\theta$ for all $k\in I\backslash\{i\}$.
By (\ref{x-yk}), there is a vector $\bfv \in V^{\kappa_0}$ such that $\bfx-\bfy=\pi_\triangle(\bfv)$ and $|\bfv|<\frac\rho{2e^{2C}}$.

By (\ref{X2Lip}), $|\zeta_\triangle^{\bfn_{ik}}.\bfv|\leq e^{\max_{j\in I}\lambda_j(\bfn_{ik})}\|\bfv\|$. Plug in (\ref{selfreturn0}); we get $|\zeta_\triangle^{\bfn_{ik}}.\bfv|\leq e^{2C}\cdot \frac\rho{2e^{2C}}=\frac\rho 2$. On the other hand, we know  \[\|\pi_\triangle(\zeta_\triangle^{\bfn_{ik}}.\bfv)\|=\|\zeta_\triangle^{\bfn_{ik}}.\bfx-\bfy\|<\theta,\]
i.e.\ there is some $\bfgamma\in\Gamma^2$ such that $|\zeta_\triangle^{\bfn_{ik}}.\bfv-\bfgamma|<\theta$. Then \[|\bfgamma|<|\zeta_\triangle^{\bfn_{ik}}.\bfv-\bfgamma|+|\zeta_\triangle^{\bfn_{ik}}.\bfv|<\frac\rho2+\theta<\frac\rho2+\frac\rho2=\rho\] and by the definition of $\rho$, it follows that  $\bfgamma=0$. Hence \begin{equation}\label{selfreturn3}|\zeta_\triangle^{\bfn_{ik}}.\bfv|<\theta,\qquad\forall k\in I\backslash\{i\}.\end{equation}

Write $\bfv$ as $\sum_{j\in I}\bfv_j$ where $\bfv_j\in V_j^{\kappa_0}$. Then by (\ref{Visquaremulti}), (\ref{selfreturn0}) and (\ref{selfreturn3})
\[\theta > |\zeta_i^{\bfn_{ik}}|\cdot|\bfv_k|=e^{\lambda_k(\bfn_{ik})}|\bfv_k|\geq e^{\frac C4}|\bfv_k|
,\]
and 
similarly 
\[\theta>|\zeta_\triangle^{\bfn_{ik}}.\bfv_i|=e^{\lambda_i(\bfn_{ik})}|\bfv_i|\geq e^C|\bfv_i|.
\]
 Thus we see
 \[
 |\bfv|^2=\sum_{j\in I}|\bfv_j|^2<e^{-\frac C2}(r_1+r_2)\theta^2<\theta^2
 \]
where the last inequality follows form $C>2\log(r_1+r_2)$. We conclude that $\bfx=\bfy+\pi_\triangle(\bfv)\in E_\theta$, which contradicts the assumption on the point~$\bfx$. Thus the claim ($\star$) is proved and the lemma follows.\end{proof}

\subsection{Applying the recurrence property to $A$}

In this part we use the recurrence property to show that when $A \supsetneq L_0$, we can find points in $A$ that deviate from $L_{0,1}$ in a highly controlled way.

\begin{lemma}\label{transdev}For all sufficiently small $\theta>0$ and all $\delta>0$, if $\kappa\in K\cup\{\infty\}$ is different from $\kappa_0$ then there are $i\in I$, $\bfy\in L_{0,1}\backslash E_\theta$ and a non-zero vector $\bfw\in V_i^\kappa$ with $|\bfw|\leq\delta$, such that $\bfy+\bfw\in A$.\end{lemma}

\begin{proof}Let $C$ and $\theta_0$ be as in Lemma \ref{selfreturn} and choose $\theta\in(0,\theta_0)$ so that the difference set $L_{0,1}\backslash E_{2\theta}$ is not empty, in which we fix a point $\bfx$.

By Condition \ref{nonisocond},  since $\bfx\in L_0$ there is a sequence of points $\bfx'_m\in A\backslash L_0$ converging to $\bfx$. We can write each $\bfx'_m$ as $\bfx+\bfw'_m$ where $\bfw'_m\in(\bR^{r_1}\oplus\bC^{r_2})^2$ and $\lim_{m\rightarrow\infty}\bfw'_m=\bfzero$.

As $\kappa\neq\kappa_0$, by Corollary \ref{Vkappatransversal}, $(\bR^{r_1}\oplus\bC^{r_2})^2=V^{\kappa_0}\oplus V^\kappa$. Thus the vector $\bfw'_m$ can be uniquely decomposed as $(\bfw'_m-\bfw_m)+\bfw_m$ where $\bfw_m\in V^\kappa$ and $\bfw'_m-\bfw_m\in V^{\kappa_0}$, moreover both $|\bfw'_m-\bfw_m|$ and $|\bfw_m|$ converge to $\bfzero$ as $m\rightarrow\infty$.

Note $\bfw_m\neq\bfzero$ because otherwise $\bfw'_m\in V^{\kappa_0}$ and $\bfx'_m$ would belong to $L_{0,1}=\bfx+T^{\kappa_0}$. By neglecting finitely many terms at the beginning of the sequence, we may assume \begin{equation}|\bfw'_m-\bfw_m|<\theta, 0<|\bfw_m|<\delta,\ \forall m.\end{equation}

$\bfw_m$ has an unique decomposition $\bfw_m=\sum_{j\in I}(\bfw_m)_j$ where $(\bfw_m)_j\in V_j^\kappa=V^\kappa\cap V_j^\square$. Then $|\bfw_m|=\big(\sum_{i\in I}|(\bfw_m)_j|^2\big)^\frac12$ and $\lim_{m\rightarrow\infty}(\bfw_m)_j=0$ for all $j$. Passing to a subsequence if neccessary, we may suppose without loss of generality that there is an index $i\in I$ such that \begin{equation}\label{transdev0}|(\bfw_m)_i|\geq|(\bfw_m)_j|,\ \forall m\in\bN,\forall j\in I\backslash\{i\}.\end{equation} In particular $|(\bfw_m)_i|>0$.

Let $\bfx_m=\bfx+(\bfw'_m-\bfw_m)$, which is in $L_{0,1}$ because $\bfw'_m-\bfw_m\in V^\kappa$. As $|\bfw'_m-\bfw_m|<\theta$ it follows from $\bfx\notin E_{2\theta}$ that $\bfx_m\notin E_\theta$.

Starting from such a point $\bfx_m$, Lemma \ref{selfreturn} allows us to construct a sequence of points $\{\bfx_{m,l}\}_{l=1}^\infty\subset L_{0,1}\backslash E_\theta$ such that $\bfx_{m,0}=\bfx_m$ and $\bfx_{m,l+1}=\zeta_\triangle^{\bfn_{m,l+1}}.\bfx_{m,l}=\zeta_\triangle^{\sum_{k=1}^{l+1}\bfn_{m,k}}.\bfx_m$ for all  $m,l\in\bN$ where $\bfn_{m,l}\in\bZ^r$ satisfies $C<\lambda_i(\bfn_{m,l})<2C$ and $\lambda_j(\bfn_{m,l})<\frac C2$ for $j\in I\backslash\{i\}$.

Construct a corresponding sequence $\{\bfw_{m,l}\}_{l=1}^\infty$ by $\bfw_{m,0}=\bfw_m$ and $\bfw_{m,l+1}=\zeta_\triangle^{\bfn_{m,l+1}}.\bfw_{m,l}=\zeta_\triangle^{\sum_{k=1}^{l+1}\bfn_{m,k}}.\bfw_m$.

For $j\in I$ let $(\bfw_{m,l})_j\in V_j^\kappa$ denote the $j$-th coordinate of $\bfw_{m,l}$. Then by (\ref{Visquaremulti}), \begin{equation}\label{transdev1}\begin{split}|(\bfw_{m,l+1})_i|=&|\zeta_i^{\bfn_{m,l+1}}|\cdot|(\bfw_{m,l})_i|=|e^{\lambda_j(\bfn_{m,l+1})}(\bfw_{m,l})_i|\\
\in& \big(e^C|(\bfw_{m,l})_i|,e^{2C}|(\bfw_{m,l})_i|\big);\end{split}\end{equation} similarly \begin{equation}\label{transdev2}|(\bfw_{m,l+1})_j|\leq  e^{\frac C2}|(\bfw_{m,l})_j|,\forall j\in I\backslash\{i\}.\end{equation}

For any $m\in\bN$, as it was previously assumed that $|\bfw_{m,0}|=|\bfw_m|$ is bounded by $\delta$, (\ref{transdev1}) implies that there is $l_m\in\bN$ such that \begin{equation}\label{transdev3}|(\bfw_{m,l_m})_i|\in[e^{-2C}\delta,\delta].\end{equation}
Furthermore,.it follows from (\ref{transdev0}), (\ref{transdev1}), (\ref{transdev2})  and (\ref{transdev3}) that for $j\neq i$,
\begin{equation}\label{transdev4}\begin{split}|(\bfw_{m,l_m})_j|\leq&\left(\frac{|(\bfw_{m,l_m})_i|}{|(\bfw_{m})_i|}\right)^\frac12\cdot|(\bfw_{m})_j|\\
\leq& \left(\frac{|(\bfw_{m,l_m})_i|}{|(\bfw_{m})_i|}\right)^\frac12\cdot|(\bfw_{m})_i|\\
\leq&\delta^{\frac 12}|(\bfw_{m})_i|^\frac12\leq \delta^{\frac 12}|\bfw_{m}|^\frac12.\end{split}\end{equation}

Let $m\rightarrow\infty$, it follows from (\ref{transdev3}) that, by taking a subsequence if neccessary, we may assume $\{(\bfw_{m,l_m})_i\}_{m=1}^\infty$ converges to a vector $\bfw\in V_i^\kappa$ with $|\bfw|\in[ e^{-2C}\delta,\delta]$. And by (\ref{transdev4}) and the fact that $\bfw_m$ converges to $0$ as $m\rightarrow\infty$, $\lim_{m\rightarrow\infty}(\bfw_{m,l_m})_j=\bfzero$ for all $j\in I\backslash\{0\}$. Therefore \begin{equation}\lim_{m\rightarrow\infty}\bfw_{m,l_m}=\lim_{m\rightarrow\infty}(\bfw_{m,l_m})_i=\bfw.\end{equation}

On the other hand, recall since $\bfx_{m,l_m}\in L_{0,1}\backslash E_\theta$ is at least of distance $\theta$ from $E$, any accumulation point cannot be in $E$. Thus by compactness of $L_{0,1}\backslash E_\theta$, by passing to a subsequence, we may suppose that $\lim_{m\rightarrow\infty}\bfx_{m,l_m}=\bfy$ for some $\bfy\in L_{0,1}\backslash E_\theta$. Then
\begin{equation}\begin{split}\bfy+\bfw=&\lim_{m\rightarrow\infty}\bfx_{m,l_m}+\lim_{m\rightarrow\infty}\bfw_{m,l_m}\\
=&\lim_{m\rightarrow\infty}\left(\zeta_\triangle^{\sum_{l=1}^{l_m}\bfn_{m,l}}.\bfx_m\right)+\lim_{m\rightarrow\infty}\left(\zeta_\triangle^{\sum_{l=1}^{l_m}\bfn_{m,l}}.\bfw_m\right)\\
=&\lim_{m\rightarrow\infty}\left(\zeta_\triangle^{\sum_{l=1}^{l_m}\bfn_{m,l}}.(\bfx_m+\bfw_m)\right)\\
=&\lim_{m\rightarrow\infty}\left(\zeta_\triangle^{\sum_{l=1}^{l_m}\bfn_{m,l}}.\big(\bfx+(\bfw'_m-\bfw_m)+\bfw_m\big)\right)\\
=&\lim_{m\rightarrow\infty}\left(\zeta_\triangle^{\sum_{l=1}^{l_m}\bfn_{m,l}}.\bfx'_m\right).
\end{split}\end{equation}

As all the $\bfx'_m$'s are in the $\zeta_\triangle$-invariant closed set $A$, this implies $\bfy+\bfw\in A$, which concludes the proof.\end{proof}

\subsection{Use of the $ r \geq 3$ assumption, and construction of extra homogeneous invariant subsets}\label{extra}
To finish the proof of Theorem \ref{CartanjoiningX2}.(1) we need the following fact:

\begin{lemma}\label{dhomononhyp}If $ r \geq 3$, $\bfy\in L_{0,1}$, $\epsilon>0$ and $i\in I$ then
\begin{equation}\label{dhomononhypeq}\overline{\{\zeta_\triangle^\bfn.\bfy:\bfn\in H, |\lambda_i(\bfn)|<\epsilon\}}=L_{0,1}\end{equation} unless $\bfy$ can be written as $\bfy_0+\bfw_0$ where $\bfy_0\in L_{0,1}$ is a torsion point and $\bfw\in V_i^{\kappa_0}$.\end{lemma}

\begin{proof} By definition of $H$ in \ref{stab01}, it stabilizes under $\zeta_\triangle$ at least one torsion point $\bfx$ from $L_{0,1}$. If $\bfy$ cannot be written in the particular form given above then $\bfy-\bfx\in T^{\kappa_0}$ cannot be written as $\bfy'+\bfw'$ where $\bfy'$ is a torsion point in $T^{\kappa_0}$ and $\bfw'\in V_i^{\kappa_0}$.

By Corollary \ref{Tkappanonhyp}, $\overline{\{\zeta_\triangle^\bfn.(\bfy-\bfx):\bfn\in H,|\lambda_i(\bfn)|<\epsilon\}}=T^{\kappa_0}$. As $\zeta_\triangle^\bfn.\bfy=\bfx+\zeta_\triangle^\bfn.(\bfy-\bfx)$ for all $\bfn\in H$ and $L_{0,1}=\bfx+T^{\kappa_0}$, (\ref{dhomononhypeq}) follows. \end{proof}

\begin{corollary}\label{transnewline}If $r\geq 3$ then there exist an index $i\in I$, a torsion point $\bfz\in L_{0,1}$ and a non-zero vector $\bfv\in V_i^\square$ such that $\bfz+\bfv\in A\backslash(\bigcup_{h=0}^qL_h)$.\end{corollary}

\begin{proof}Let $\theta$ be sufficiently small then $L_{0,1}\backslash E_\theta$ is non-empty and has open interior with respect to the relative topology in $L_{0,1}$. In particular, $L_{0,1}\backslash E_\theta$ contains torsion points because torsion points are dense in $L_{0,1}$.

Notice $L_{0,1}\backslash E_\theta$ is disjoint from all the $L_{0,t}$'s for $1<t\leq s_0$ and all the $L_h$'s for $1\leq h\leq q$. As each of these sets are compact, it follows \begin{equation}\dist\Big(L_{0,1}\backslash E_{\theta},(\bigcup_{t=2}^{s_0}L_{0,t})\cup(\bigcup_{h=1}^qL_h)\Big)>0.\end{equation}
Define a positive number \begin{equation}\delta=\frac12\min\left(\dist\Big(L_{0,1}\backslash E_{\theta},(\bigcup_{t=2}^{s_0}L_{0,t})\cup(\bigcup_{h=1}^qL_h)\Big),\rho_0\right).\end{equation}

Applying Lemma \ref{transdev}, we obtain $i\in I$, $\bfy\in L_{0,1}\backslash E_\theta$ and $\bfw\in V_i^\kappa$ where $\kappa\in K\cup\{\infty\}$ is distinct from $\kappa_0$, such that $0<|\bfw|<\delta$ and $\bfy+\bfw\in A$.

We distinguish between two cases:\\

\noindent{\bf Case 1.} If $\bfy$ can be decomposed as $\bfz+\bfw'$ where $\bfz$ is a torsion point from $L_{0,1}$ and $\bfw'\in V_i^{\kappa_0}$. Then $\bfy+\bfw$ rewrites as $\bfz+\bfv$ where $\bfv=\bfw'+\bfw$ is in $V_i^\square$ since both $\bfw'$ and $\bfw$ are.\\

\noindent{\bf Case 2.} Suppose $\bfy$ cannot be decomposed as above, then by Lemma \ref{dhomononhyp}, \[\{\zeta_\triangle^\bfn.\bfy:\bfn\in H,|\lambda_i(\bfn)|<\frac12\}\] is dense in $L_{0,1}$. In particular, we fix an arbitrary torsion point $\bfz\in L_{0,1}\backslash E_\theta$, then there is a sequence $\{\bfn_m\in H\}_{m=1}^\infty$ such that $|\lambda_i(\bfn_m)|<\frac12$ for all $m$ and $\lim_{m\rightarrow\infty}\zeta_\triangle^{\bfn_m}.\bfy=\bfz$. Then for each $m$, $\zeta_\triangle^{\bfn_m}.\bfw=\zeta_i^{\bfn_m}\bfw\in V_i^\kappa$ and $|\zeta_\triangle^{\bfn_m}.\bfw|=e^{\lambda_i(\bfn_m)}|\bfw|\in [e^{-\frac12}\delta,e^\frac12\delta]$. Since $\{\bfv\in V_i^\kappa:|\bfv|\in[e^{-\frac12}\delta,e^\frac12\delta]\}$ is compact, without loss of generality we may assume $\zeta_\triangle^{\bfn_m}.\bfw$ converges to a vector $\bfv\in V_i^\kappa$ with $|\bfv|\in[e^{-\frac12}\delta,e^\frac12\delta]$ as $m$ tends to $\infty$. In particular $0<|\bfv|<2\delta$. Then $\bfz+\bfv$ is in $A$ as it is the limit of $\zeta_\triangle^{\bfn_m}.(\bfy+\bfw)$.\\

So in both cases $A$ contains $\bfz+\bfv$ where $\bfz\in L_{0,1}$ is of torsion and $\bfv\in V_i^\square$ is non-trivial. Moreover, $\bfz+\bfv$ can always be written as $\bfz'+\bfv'$ where $\bfz'\in L_{0,1}\backslash E_\theta$ (though not necessarily a torsion point) and $\bfv'$ is a non-zero vector from $V_i^\kappa$ with $|\bfv'|<2\delta$ (in Case 1, $\bfz'=\bfy$ and $\bfv'=\bfw$; while in Case 2 $\bfz'=\bfz$ and $\bfv'=\bfv$.)

It remains to show $\bfz+\bfv\notin L_h$ for all $h=0,\cdots,q$.
On the one hand, by the definition of $\delta$,\begin{equation}|\bfv'|<\dist\Big(L_{0,1}\backslash E_{\theta},(\bigcup_{t=2}^{s_0}L_{0,t})\cup(\bigcup_{h=1}^qL_h)\Big),\end{equation} so $\bfz+\bfv\notin (\bigcup_{t=2}^{s_0}L_{0,t})\cup(\bigcup_{h=1}^qL_h)$; on the other hand because $0<|\bfv'|< \rho_0$ and $\bfv'\notin V^{\kappa_0}$, $\bfz'+\bfv'$ is not in $L_{0,1}$. Therefore $\bfz+\bfv\notin\bigcup_{h=1}^qL_h$, which concludes the proof.\end{proof}

Now we are ready to prove Lemma \ref{nonisoinduction}.

\begin{proof}[Proof of Lemma \ref{nonisoinduction}] As $r\geq 3$, by Corollary \ref{transnewline}, $A$ contains $\bfz$ and $\bfz+\bfv$ where $\bfz$ is a torsion point and $\bfv$ is a non-zero vector from one of the $V_i^\square$'s such that $\bfz+\bfv\notin L_h, \forall h=0,\cdots q$. By Lemma \ref{denseorcontaindhomo} it contains a $d$-dimensional homogeneous $\zeta_\triangle$-invariant subset $L_{q+1}$ such that $\bfz+\bfv\in L_{q+1}$. It follows that $L_{q+1}$ is different from all the previous $L_h$'s.\end{proof}

Hence so far we have completed the proof of Proposition \ref{nonisofull}, which is going to be used to establish both Theorems \ref{CartanjoiningX2}.(1) and \ref{nondenserationalX2}.

\subsection{Proof of rigidity results}

\begin{proof}[Proof of Theorem \ref{CartanjoiningX2}.(1)]Let $A$ be an infinite proper closed subset of $X^2$, invariant and topologically transitive under the action $\zeta_\triangle$.

Assuming $A$ is not a $d$-dimensional homogeneous $\zeta_\triangle$-invariant subset, we try to deduce a contradiction.

We claim $A$ meets Condition \ref{nonisocond}. To see this, recall by Proposition \ref{firstdhomo}, $A$ always contains a $d$-dimensional homogeneous $\zeta_\triangle$-invariant subset $L$. Then $A \backslash L$ is non-empty and relatively open in $A$. Furthermore as both $A$ and $L$ are $\zeta_\triangle$-invariant, so is $A \backslash L$. Let $U$ be any relatively open subset in $A$. By topological transitivity, there exists an $\bfn \in\bZ^r$ such that $\big(\zeta_\triangle^\bfn.(A \backslash L)\big)\cap U$ is not empty. But $\zeta_\triangle^\bfn.(A \backslash L)$ is just $A \backslash L$. This shows $\overline{A \backslash L}=A$.

As $r \geq 3$, Proposition \ref{nonisofull} applies and $A=X^2$. But $A$ is supposed to be proper and we get a contradiction.

Moreover, suppose $A$ is a $\zeta _ \triangle$-invariant, closed, proper subset of $X ^ 2$ (not necessarily topologically transitive). Then by Lemma~\ref{finitedhomo}, there can be only finitely many $d$-dimensional homogeneous invariant subsets $L _ 0$,\dots,$L _ q \subset A $. If $A'=A \setminus \bigcup _ {i = 0} ^ q L _ i$ is finite, this shows that $A$ is a finite union of topologically transitive invariant subsets. If $A '$ is infinite, by Proposition~\ref{firstdhomo} there is a $d$-dimensional homogeneous invariant subset $L \subset \overline {A '}$. It follows that $\overline {A '}$ satisfies Condition \ref{nonisocond}, hence by Proposition~\ref{nonisofull} we have that $\overline {A '} = X ^ 2$ --- a contradiction.
\end{proof}

\begin{proof}[Proof of Theorem \ref{nondenserationalX2}] Fix $\epsilon>\epsilon'>0$ and $\bfx\in X^2$. Let $C_\bfx\subset X^2$ be the set of all torsion points whose $\zeta_\triangle$-orbits are disjoint from $\mathring B_{\epsilon'}(\bfx)$ and notice it is $\zeta_\triangle$-invariant. We want to show $C_\bfx$ is covered by a given finite union of $d$-dimensional homogeneous $\zeta_\triangle$-invariant sets.

Take all the $d$-dimensional homogeneous $\zeta_\triangle$-invariant subsets $L_1,\cdots,L_q$ which are not $\epsilon'$-dense. Write \begin{equation}A_\bfx=\overline{C_\bfx\backslash{\bigcup}_{h=1}^qL_h}.\end{equation}
Then as $C_\bfx$ and the $L_h$'s are all $\zeta_\triangle$-invariant, so are both $C_\bfx\backslash\bigcup_{h=1}^qL_h$ and $A_\bfx$. Moreover, $A_\bfx$ is disjoint from $\mathring B_{\epsilon'}(\bfx)$.

We claim $A_\bfx$ is finite. Suppose not, then $A_\bfx$ contains a $d$-dimensional homogeneous $\zeta_\triangle$-invariant subset $L$ by Proposition \ref{firstdhomo}. As $L$ avoids $\mathring B_{\epsilon'}(\bfx)$, it must be one of $L_1,\cdots,L_q$. In particular, \begin{equation}A_\bfx\backslash L\supset (C_\bfx\backslash{\bigcup}_{h=1}^qL_h)\backslash L=C_\bfx\backslash{\bigcup}_{h=1}^qL_h\end{equation} and it follows that $\overline{A_\bfx\backslash L}\supset \overline{C_\bfx\backslash{\bigcup}_{h=1}^qL_h}=A_\bfx$. So $A_\bfx$ satisfies Condition \ref{nonisocond}. Because $r\geq 3$, $A_{\epsilon',\bfx}=X^2$ by Proposition \ref{nonisofull}, which contradicts the fact that $A$ is disjoint from $\mathring B_{\epsilon'}(\bfx)$. It follows \begin{equation}\label{nondenserationalX20}\big|C_\bfx\backslash{\bigcup}_{h=1}^qL_h\big|<\infty, \forall \bfx\in X^2.\end{equation}

By compactness of $X^2$, there are $\bfx_1,\cdots,\bfx_l$ such that $\bigcup_{k=1}^l\mathring B_{\epsilon-\epsilon'}(\bfx_k)$ covers $X^2$. 

Now let $\bfy\in X^2$ be a torsion point whose orbit $\{\zeta_\triangle^\bfn.\bfy:\bfn\in\bZ^r\}$ is not $\epsilon$-dense, or equivalently, is disjoint from $\mathring B_\epsilon(\bfx)$ for some $\bfx\in X$. Take $k\in\{1,\cdots,l\}$ such that $\bfx\in\mathring B_{\epsilon-\epsilon'}(\bfx_k)$. Then the orbit of $\bfy$ is also disjoint from $\mathring B_{\epsilon'}(\bfx_k)$. So by construction the set $C_{\bfx_k}$ contains $\bfy$.

Therefore the collection of all torsion points whose orbits are not $\epsilon$-dense, which we denote by $C$, is contained in the set \begin{equation}{\bigcup}_{k=1}^lC_{\bfx_k}\subset({\bigcup}_{h=1}^q{L_h})\sqcup\big({\bigcup}_{k=1}^l(C_{\bfx_k}\backslash{\bigcup}_{h=1}^qL_h)\big).\end{equation} This proves the first part of theorem because $\big({\bigcup}_{k=1}^l(C_{\bfx_k}\backslash{\bigcup}_{h=1}^qL_h)\big)$ is finite by (\ref{nondenserationalX20}).

To show the last claim in Theorem \ref{nondenserationalX2}, i.e. $C$ is covered by finitely many $d$-dimensional homogeneous $\zeta_\triangle$-invariant subsets, it suffices to note that every torsion point in $X^2$ is contained in some $d$-dimensional homogeneous $\zeta_\triangle$-invariant subset. We assign to each point in the finite set $\big({\bigcup}_{k=1}^l(C_{\bfx_k}\backslash{\bigcup}_{h=1}^qL_h)\big)$ a $d$-dimensional homogeneous $\zeta_\triangle$-invariant subset that contains it. Then all these subsets together with $L_1,\cdots, L_q$ cover $C$.\end{proof}

So far we established Theorem \ref{CartanjoiningX2}.(1) as well as Theorem \ref{nondenserationalX2}. By Lemma \ref{CartanjoiningT2X2}, Theorem \ref{Cartanjoining}.(1) and Theorem \ref{nondenserational} follow.

\section{Counterexample in rank 2}\label{rank2}

We now construct the counterexample required by the second part of Theorem \ref{CartanjoiningX2}. From now on let $r=2$.

In this case, $r_1+r_2=r+1=3$ so $I=\{1,2,3\}$ and $\bR^{r_1}\oplus\bC^{r_2}=V_1\oplus V_2\oplus V_3$. Fix a non-trivial lattice point $\gamma\in\Gamma\backslash\{0\}$ and decompose it as $\sum_{i=1}^3\gamma_i$ where $\gamma_i\in V_i$. Recall $\gamma=\sigma(\theta)$ for some $\theta\in K$ as $\Gamma\subset\sigma(K)$. Moreover, $\theta\neq 0$ as otherwise $\gamma=0$. Hence for each $i$, $\gamma_i$ doesn't vanish as it corresponds to $\sigma_i(\theta)$ in $V_i$.

Take the point $\tilde\bfx=(\gamma_1,\gamma_2)\in(\bR^{r_1}\oplus\bC^{r_2})^2$. And let $\bfx=\pi_\triangle(\tilde\bfx)=\big(\pi(\gamma_1),\pi(\gamma_2)\big)$. We will prove in this section that (\ref{CartanjoiningX2eq}) holds with \begin{equation}\kappa_1=\infty,\ \kappa_2=0\text{ and }\kappa_3=-1.\end{equation}

From Lemma \ref{Tkappatangent} one can easily check:
\begin{equation}T^\infty=\pi_\triangle(V^\infty)=\pi_\triangle\big(\{0\}\times(\bR^{r_1}\oplus\bC^{r_2})\big)=\{0\}\times X;\end{equation}
\begin{equation}T^0=\pi_\triangle(V^0)=\pi_\triangle\big((\bR^{r_1}\oplus\bC^{r_2})\times\{0\}\big)=X\times\{0\};\end{equation}
and
\begin{equation}\begin{split}T^{-1}=&\pi_\triangle(V^{-1})=\pi_\triangle\big(\{(\tilde x,-\tilde x):\tilde x\in (\bR^{r_1}\oplus\bC^{r_2})\}\big)\\
=&\{(x,-x):x\in X\}.\end{split}\end{equation}

The point $\bfx$ is special in the following sense:

\begin{lemma}\label{homoclinic} For each $i\in\{1,2,3\}$, there is a permutation $(i,j,h)$ of $\{1,2,3\}$ such that $\bfx$ can be decomposed as $\pi_\triangle(\tilde\bfx_\parallel)+\pi_\triangle(\tilde\bfx_\top)$ where $\tilde\bfx_\parallel$ and $\tilde\bfx_\top$ are non-zero vectors, respectively from $V_j^{\kappa_i}$ and $V_i^{\kappa_j}$. In particular, $\pi_\triangle(\tilde\bfx_\parallel)\in T^{\kappa_0}$.
\end{lemma}

\begin{proof}We choose $j$, $h$, $\tilde\bfx_\parallel$ and $\tilde\bfx_\top$ in the following way : \begin{itemize}\item if $i=1$, let $j=2$, $h=3$, $\tilde\bfx_\parallel=(0,\gamma_2)$ and $\tilde\bfx_\top=(\gamma_1,0)$; \item if $i=2$, let $j=1$, $h=3$, $\tilde\bfx_\parallel=(\gamma_1,0)$ and $\tilde\bfx_\top=(0,\gamma_2)$; \item if $i=3$, let $j=1$, $h=2$, $\tilde\bfx_\parallel=(\gamma_1,-\gamma_1)$ and $\tilde\bfx_\top=(0,-\gamma_3)$.\end{itemize}

The lemma is obviously true for $i=1,2$ as $\tilde\bfx_\parallel=\tilde\bfx_\parallel+\tilde\bfx_\top$. For $i=3$, notice $\tilde\bfx=\tilde\bfx_\parallel+\tilde\bfx_\top+(0,\gamma)$. Thus $\bfx=\pi_\triangle(\tilde\bfx)=\pi_\triangle(\tilde\bfx_\parallel)+\pi_\triangle(\tilde\bfx_\top)$ as $(0,\gamma)\in\Gamma^2$.
\end{proof}

\begin{corollary}\label{counterapproach}For each $i=1,2,3$, suppose a sequence $\{\bfn_l\}_{l=1}^\infty\subset\bZ^2$ satisfies that $\lim_{l\rightarrow\infty}\lambda_i(\bfn_l)=-\infty$ and $\lim_{l\rightarrow\infty}\zeta_\triangle^{\bfn_l}.\bfx=\bfy$, then $\bfy\in T^{\kappa_i}$.\end{corollary}
\begin{proof} Take the decomposition given by Lemma \ref{homoclinic}, then $\zeta_\triangle^{\bfn_l}.\bfx=\zeta_\triangle^{\bfn_l}.\big(\pi_\triangle(\tilde\bfx_\parallel)+\pi_\triangle(\tilde\bfx_\top)\big)=\zeta_\triangle^{\bfn_l}.\pi_\triangle(\tilde\bfx_\parallel)+\pi_\triangle(\zeta_\triangle^{\bfn_l}.\tilde\bfx_\top)$. Furthermore, by (\ref{Visquaremulti}), $|\zeta_\triangle^{\bfn_l}.\tilde\bfx_\top|=e^{\lambda_i(\bfn_l)}|\tilde\bfx_\top|$. When $l$ go to $\infty$, $e^{\lambda_i(\bfn_l)}\rightarrow 0$ and thus $\pi_\triangle(\zeta_\triangle^{\bfn_l}.\tilde\bfx_\top)$ approaches $\bfzero$. Hence  $\lim_{l\rightarrow\infty}\zeta_\triangle^{\bfn_l}.\bfx=\lim_{l\rightarrow\infty}\zeta_\triangle^{\bfn_l}.\pi_\triangle(\tilde\bfx_\parallel)$ and as $\pi_\triangle(\tilde\bfx_\parallel)\in T^{\kappa_i}$ this limit must be in $T^{\kappa_i}$.\end{proof}

\begin{corollary}\label{countersmall}Any accumulation point of $\{\zeta_\triangle^\bfn.\bfx:\bfn\in\bZ^r\}$ is contained in at least one of $T^\infty$, $T^0$ and $T^{-1}$.\end{corollary}
\begin{proof} For each $i=1,2,3$, we define a subset $\Phi_i$ of $W$ by:
\begin{equation}\Phi_i=\big\{w\in\bR^3:\sum_{k=1}^3d_kw_k=0,w_i<0,w_k>w_i\mathrm{\ for\ }k\neq i\big\}.\end{equation}
Then $\cup_{i=1}^3\Phi_i=W$. Therefore by Remark \ref{dirichlet}, for any accumulation point $\bfy$ of $\{\zeta_\triangle^\bfn.\bfx:\bfn\in\bZ^r\}$, there is $i\in\{1,2,3\}$ and a sequence $\{\bfn_l\}_{l=1}^\infty$ where all the $\bfn_l$'s are different from each other, such that $\zeta_\triangle^{\bfn_l}.\bfx$ converges to $\bfy$ as $l\rightarrow\infty$ and $\cL(\bfn_l)\in\Phi_i$ for all $l$.

For any $R>0$, it is easy to check that $\Phi_{i,R}=\{w\in \Phi_i:w_i\geq -R\}$ is a compact subset. Because $\cL$ embedds $\bZ^r$ as a full-rank lattice into $W$, there are only finitely many $\bfn\in\bZ^r$ such that $\cL(\bfn)\in\Phi_{i,R}$ . Therefore for sufficiently large $l$, $\cL(\bfn_l)\in\Phi_i\backslash\Phi_{i,R}=\{w\in \Phi_i:w_i<-R\}$. Equivalently, the $i$-th coordinate of $\cL(\bfn_l)$, which is just $\lambda_i(\bfn_l)$, is strictly less than $-R$.  Since $R$ is arbitrary,  we see \begin{equation}\label{counterapproach0}\lim_{l\rightarrow\infty}\lambda_i(\bfn_l)=-\infty.\end{equation} Hence by Corollary \ref{counterapproach}, $\bfy\in T^{\kappa_i}$.\end{proof}

\begin{lemma}\label{counterlarge}For each $i=1,2,3$, $T^{\kappa_i}\subset\overline{\{\zeta_\triangle^\bfn.\bfx:\bfn\in\bZ^r\}}$.\end{lemma}

\begin{proof} Let $(i,j,h)$, $\bfx_\parallel$ and $\bfx_\bot$ be given by Lemma \ref{homoclinic}. Because $\cL(\bZ^2)$ is a full-rank lattice in $W$, there is $\bfn\in\bZ^2$ such that $\lambda_i(\bfn)<0$ and $\lambda_j(\bfn),\lambda_h(\bfn)>0$. 

As $X^2$ is compact, for some subsequence of integers $\{l_k\}_{k=1}^\infty$ of $\bN$, $\zeta_\triangle^{l_k\bfn}.\bfx$ converges to a limit $\bfy$. Because $\lambda_i(l_k\bfn)=l_k\lambda_i(\bfn)$ coverges to $-\infty$ as $k\rightarrow\infty$, it follows from Lemma \ref{Tkappatangent} that $\bfy\in T^{\kappa_i}$. In particular, there is at least one accumulation point of $\{\zeta_\triangle^\bfn.\bfx:\bfn\in\bZ^r\}$ in $T^{\kappa_i}$.

So we can consider the non-empty intersection $E=T^{\kappa_i}\cap\overline{\{\zeta_\triangle^\bfn.\bfx:\bfn\in\bZ^r\}}$. As both $T^\kappa$ and $\overline{\{\zeta_\triangle^\bfn.\bfx:\bfn\in\bZ^r\}}$ are $\zeta_\triangle$-invariant, $E$ is a $\zeta_\triangle$-invariant closed subset in $T^{\kappa_i}$.

To establish the lemma it suffices to show $E=T^{\kappa_i}$. Thanks to Lemma \ref{TkappaBerend}, this is equivalent to showing that $E$ is infinite.

In order to deduce a contradiction, assume $E$ is finite, then it must consisits of torsion points. In particular, the point $\bfy$ constructed above is of torsion and therefore the subgroup $H=\Stab_{\zeta_\triangle}(\bfy)$ has finite index in $\bZ^2$. So there is $p\in\bN$ such that $p\bfn\in H$. Denote \begin{equation}\label{counterlarge2}C=\max\big(p\lambda_j(\bfn),p\lambda_h(\bfn)\big).\end{equation}

Since $E$ is a finite set, similar to (\ref{selfreturn2}), there is a constant $\rho$ such that for all non-zero vector $\bfv\in V^{\kappa_i}$ with $|\bfv|<\rho$ we have $\bfy+\bfv\in T^{\kappa_i}\backslash E$.

Define a shell-shaped set \begin{equation}\label{counterlarge3}D=\big\{\bfy+\bfv:\bfv\in V^{\kappa_i}, |\bfv|\in\big[\frac\rho{2e^C},\frac\rho2\big]\big\},\end{equation} which is a compact subset of $T^{\kappa_i}$ and is disjoint from $E$. Then by construction of $E$, the set $D$ is disjoint from $\overline{\{\zeta_\triangle^\bfn.\bfx:\bfn\in\bZ^r\}}$ as well.
Set \begin{equation}\delta=\min\big(\dist(D, \overline{\{\zeta_\triangle^\bfn.\bfx:\bfn\in\bZ^r\}}),\frac12\rho\big).\end{equation} 

For any $\bfz\in D$, because $\bfz\in\mathring B_{\frac{\delta_{\bfz_m}}2}(\bfz_m)$ for some $m$, $B_\delta(\bfz)$ is contained in $B_{\delta+\frac{\delta_{\bfz_m}}2}(\bfz_m)\subset B_{\delta_{\bfz_m}}(\bfz_m)$. Thus \begin{equation}\label{counterlarge4}B_\delta(\bfz) \cap\{\zeta_\triangle^\bfn.\bfx:\bfn\in\bZ^r\}=\emptyset, \forall \bfz\in D\end{equation}

Following earlier discussions, we know for any sufficiently large integer $k$, \begin{equation}\label{paralleltopclose}\|\pi_\triangle(\zeta_\triangle^{l_k\bfn}.\bfx_\parallel)-\bfy\|<\delta,\  |\zeta_\triangle^{l_k\bfn}.\bfx_\top|<\delta.\end{equation}

Fix such a large $k$, then there is a vector $\bfw\in V^{\kappa_i}$ such that $|\bfw|<\delta$ and $\pi_\triangle(\zeta_\triangle^{l_k\bfn}.\bfx_\parallel)=\bfy+\bfw$. Write $\bfw$ as $\bfw_i+\bfw_j+\bfw_h$ where $\bfw_i\in V_i^{\kappa_i}$, $\bfw_j\in V_j^{\kappa_i}$ and $\bfw_h\in V_h^{\kappa_i}$. 

Since $\bfy=\pi_\triangle(\zeta_\triangle^{l_k\bfn}.\bfx_\parallel-\bfw)$ is a torsion point, there is a non-zero integer $q$ such that $q\bfy=\pi_\triangle\big(q(\zeta_\triangle^{l_k\bfn}.\bfx_\parallel-\bfw)\big)=\bfzero$ or equivalently $q(\zeta_\triangle^{l_k\bfn}.\bfx_\parallel-\bfw)\in\Gamma^2$. Because $\bfx_\parallel$ is a non-zero element from $V_j^{\kappa_i}$ and $\lambda_j(\bfn)>0$, $|\zeta_\triangle^{l_k\bfn}.\bfx_\parallel|>\delta$ if $k$ is chosen to be sufficiently large . But $|\bfw|<\delta$, so $q(\zeta_\triangle^{l_k\bfn}.\bfx_\parallel-\bfw)\in\Gamma^2\backslash\{\bfzero\}$. Recall $\Gamma\subset\sigma(K)$ in the setting of this paper. Thus $q(\zeta_\triangle^{l_k\bfn}.\bfx_\parallel-\bfw)=\big(\sigma(\theta^{(1)}),\sigma(\theta^{(2)})\big)$ where $\theta^{(1)}, \theta^{(2)}$ are from $K$ and are not simultaneously zero. In particular, the $V_h^{(1)}$ and $V_h^{(2)}$ coordinates of $q(\zeta_\triangle^{l_k\bfn}.\bfx_\parallel-\bfw)$ are respectively $\sigma_h(\theta^{(1)})$ and $\sigma_h(\theta^{(2)})$, and are not simultaneously zero. In other words, in the decomposition \begin{equation}q(\zeta_\triangle^{l_k\bfn}.\bfx_\parallel-\bfw)=q(\zeta_\triangle^{l_k\bfn}.\bfx_\parallel-\bfw_j)+q\bfw_i+q\bfw_h,\end{equation} the $V_h^\square$-component $q\bfw_h$ doesn't vanish. Thus $\bfw_h\neq \bfzero$.

Now consider the sequence of vectors $\bfv_s=\zeta_\triangle^{sp\bfn}.\bfw\in V^{\kappa_i}$ where $s=1,2,\cdots$. Note $\bfv_s$ decomposes into $(\bfv_s)_i+(\bfv_s)_j+(\bfv_s)_h$ where  $(\bfv_s)_f=\zeta_\triangle^{sp\bfn}.\bfw_f\in V_f^{\kappa_i}$ for $f=i,j,h$. By (\ref{Visquaremulti}), $|(\bfv_s)_f|=e^{sp\lambda_i(\bfn)}|\bfw_f|$, in particular $|(\bfv_s)_f|\leq e^{p\lambda_f(\bfn)}|(\bfv_{s-1})_f|\leq e^C|(\bfv_{s-1})_f|$ for all $s\geq 1$ and $ f=i,j,h$ because of (\ref{counterlarge2}) and the fact that $\lambda_i(\bfn)<0$. Because $|\bfv_s|^2=|(\bfv_s)_i|^2+|(\bfv_s)_i|^2+|(\bfv_s)_i|^2$, we have \begin{equation}|\bfv_s|\leq e^C|\bfv_{s-1}|, \forall s\geq 1.\end{equation}
But on the other hand because $\bfw_h\neq\bfzero$ and $\lambda_h(\bfn)>0$,   as $s$ increases $|(\bfv_s)_h|=e^{sp\lambda_i(\bfn)}|\bfw_h|$ grows to $\infty$, in consequence $|\bfv_s|\rightarrow\infty$ as well. Since $|\bfv_0|=|\bfw|\leq\delta\leq\frac\rho 2$, the above facts imply that there is an integer $s_0\geq 0$ such that \begin{equation}\label{counterlarge5}|\bfv_{s_0}|\in\big[\frac\rho{2e^C},\frac\rho2\big].\end{equation}

We investigate the point \begin{equation}\label{farpoint}\zeta_\triangle^{(l_k+s_0p)\bfn}.\bfx= \zeta_\triangle^{s_0p\bfn}.\zeta_\triangle^{l_k\bfn}.\bfx
=\zeta_\triangle^{s_0p\bfn}.\pi_\triangle(\zeta_\triangle^{l_k\bfn}.\tilde\bfx_\parallel)+\zeta_\triangle^{s_0p\bfn}.(\zeta_\triangle^{l_k\bfn}.\tilde\bfx_\top).\end{equation}

Observe $\zeta_\triangle^{s_0p\bfn}.\pi_\triangle(\zeta_\triangle^{l_k\bfn}.\tilde\bfx_\parallel)=\zeta_\triangle^{s_0p\bfn}.\bfy+\zeta_\triangle^{s_0p\bfn}.\bfw$.

First, $\zeta_\triangle^{s_0p\bfn}.\bfy=\bfy$ because $s_0p\bfn$ belongs to $H$ and thus stabilizes $\bfy$ under $\zeta_\triangle$. Furthermore, $\zeta_\triangle^{s_0p\bfn}.\bfw=\bfv_{s_0}\in V^{\kappa_i}$. So it follows from (\ref{counterlarge3}) and (\ref{counterlarge5}) that \begin{equation}\label{counterlarge6}\zeta_\triangle^{s_0p\bfn}.\pi_\triangle(\zeta_\triangle^{l\bfn}.\tilde\bfx_\parallel)\in D.\end{equation}

On the other hand, since $\lambda_i(\bfn)<0$ and $\zeta_\triangle^{l\bfn}.\tilde\bfx_\top\in V_i^{\kappa_j}$, by (\ref{paralleltopclose}) \begin{equation}|\zeta_\triangle^{s_0p\bfn}.(\zeta_\triangle^{l\bfn}.\tilde\bfx_\top)|\leq e^{s_0p\lambda_i(\bfn)}|\zeta_\triangle^{l\bfn}.\tilde\bfx_\top|\leq |\zeta_\triangle^{l\bfn}.\tilde\bfx_\top|\leq\delta.\end{equation} From (\ref{counterlarge5}) and (\ref{farpoint}) we deduce that $\zeta_\triangle^{(l_k+s_0p)\bfn}.\bfx\in B_\delta\big(\zeta_\triangle^{s_0p\bfn}.\pi_\triangle(\zeta_\triangle^{l\bfn}.\tilde\bfx_\parallel)\big)$. But, together with (\ref{counterlarge6}), this contradicts the disjointness relation (\ref{counterlarge4}).

So $E$ cannot be a proper subset of $T^{\kappa_i}$ and the lemma follows.\end{proof}

\begin{proof}[Proof of Theorem \ref{CartanjoiningX2}.(2)]The second half of the theorem directly follows from Corollary \ref{countersmall} and Lemma \ref{counterlarge}.\end{proof}

So finally we established Theorem \ref{CartanjoiningX2} and, by Lemma \ref{CartanjoiningT2X2}, the main result Theorem \ref{Cartanjoining} as well.\newline

{\noindent\bf Acknowledgments:} The authors are grateful to Jean Bourgain and Bryna Kra for helpful discussions. We thank the referees for valuable comments and suggestions. Z.W. wishes to thank the Hebrew University of Jerusalem for its hospitality.

\end{document}